\documentclass[12pt,oneside]{amsart}

\usepackage{mathpazo}
\usepackage[normalem]{ulem}

\usepackage{bbold}
\usepackage{latexsym}
\usepackage{amssymb} 
\usepackage{mathrsfs}  \usepackage[usenames,dvipsnames]{xcolor}
\usepackage{hyperref}
\usepackage{stmaryrd}
\usepackage{caption}
\usepackage{tikz-cd}
\usepackage{mathtools}
\usepackage{bm}
\usepackage[cmtip,all,matrix,arrow,tips,curve]{xy}
\usepackage{comment}
\usepackage{enumitem}
\usetikzlibrary{matrix,arrows,backgrounds,shapes.misc,shapes.geometric,fit}
\usetikzlibrary{arrows,decorations.pathmorphing,backgrounds,positioning,fit,petri,calc,shapes.misc,decorations.markings}
\usepackage{adjustbox}
\hypersetup{
         unicode=false,               pdftoolbar=true,             pdfmenubar=true,             pdffitwindow=false,          pdfstartview={FitH},         pdftitle={Essential torsion finiteness},
        pdfnewwindow=true,           colorlinks=true,            linkcolor=Sepia,             citecolor=Red,             filecolor=magenta,           urlcolor=cyan            }

\let\oldtocsection=\tocsection

\let\oldtocsubsection=\tocsubsection

\let\oldtocsubsubsection=\tocsubsubsection

\renewcommand{\tocsection}[2]{\hspace{0em}\oldtocsection{#1}{#2}}
\renewcommand{\tocsubsection}[2]{\hspace{2em}\oldtocsubsection{#1}{#2}}
\renewcommand{\tocsubsubsection}[2]{\hspace{4em}\oldtocsubsubsection{#1}{#2}}

\newtheorem{thm}{Theorem}[section]
\newtheorem{prop}[thm]{Proposition}
\newtheorem{lemma}[thm]{Lemma}
\newtheorem{cor}[thm]{Corollary}
\newtheorem{conj}[thm]{Conjecture}
 \newtheorem{defn}[thm]{Definition}
 
\newtheorem{ques}[thm]{Question}

\newtheorem*{ques*}{Question}

\theoremstyle{remark}
\newtheorem{remark}[thm]{Remark}
\newtheorem{exa}[thm]{Example}

\numberwithin{equation}{section}

\makeatletter
\def\imod#1{\allowbreak\mkern5mu{\operator@font mod}\,\,#1}
\makeatother

\def\makeop#1{\expandafter\def\csname#1\endcsname
  {\mathop{\rm #1}\nolimits}\ignorespaces}
\makeop{Hom}   \makeop{End}   \makeop{Aut}   \makeop{Isom}  \makeop{Pic}
\makeop{Gal}   \makeop{ord}   \makeop{Char}  \makeop{Div}   \makeop{Lie}
\makeop{PGL}   \makeop{Corr}  \makeop{PSL}   \makeop{sgn}   \makeop{Spf}
\makeop{Spec}  \makeop{Map}    \makeop{Nm}    \makeop{Fr}    \makeop{disc}
\makeop{Proj}  \makeop{supp}  \makeop{ker}   \makeop{im}    \makeop{dom}
\makeop{coker} \makeop{Stab}  \makeop{SO}    \makeop{SL}    \makeop{SL} \makeop{O}
\makeop{Cl}    \makeop{cond}  \makeop{Br}    \makeop{inv}   \makeop{rank}
\makeop{id}    \makeop{Fil}   \makeop{Frac}   \makeop{SU}
\makeop{Trd}   \makeop{Sp}    \makeop{GSp}   \makeop{Tr}    \makeop{Trd}   
\makeop{Res}   \makeop{ind}   \makeop{depth} \makeop{diag}    \makeop{st}
\makeop{Ad}    \makeop{Int}   \makeop{tr}    \makeop{Sym}   \makeop{can}
\makeop{length}\makeop{SO}    \makeop{torsion} \makeop{GSp} \makeop{Ker}
\makeop{Adm}   \makeop{Frob}  \makeop{id}    \makeop{Tor}   \makeop{Ind}
\makeop{CoInd} \makeop{Inf}   \makeop{Cor}   \makeop{CoInf} \makeop{coLie}
\makeop{Ann} \makeop{rec} \makeop{Image} \makeop{Rep} \makeop{Inn} \makeop{Out}
\makeop{Def} \makeop{B} \makeop{gr} 
\makeop{Art}
\makeop{codim} \makeop{T} 
\makeop{tors}
\makeop{cyc}

\DeclareMathOperator{\res}{Res}
\newcommand{\one}{{\mathbb 1}}
\newcommand{\Z}{\mathbb Z}
\newcommand{\Q}{\mathbb Q}

\newcommand{\C}{\mathbb C}

                  \newcommand{\G}{\mathbb G}

\newcommand{\OO}{\mathcal O}

\DeclareMathOperator{\spec}{Spec}

 \newcommand{\rat}{\mathbb Q}
\newcommand{\integ}{\mathbb Z}
\newcommand{\real}{\mathbb R}
\newcommand{\cx}{\mathbb C}
\newcommand{\gp}{\mathbb G}
\newcommand{\ff}{\mathbb F}
\newcommand{\ang}[1]{{{\langle #1 \rangle}}}
\def\inv{^{-1}}
\def\units{^\times}
\def\ab{{\rm ab}}

\DeclareMathOperator{\aut}{Aut}
\DeclareMathOperator{\gal}{Gal}

\def\iso{\cong}
\DeclareMathOperator{\GL}{GL}
\DeclareMathOperator{\gu}{GU}
\DeclareMathOperator{\su}{SU}
\DeclareMathOperator{\fix}{Fix}
\DeclareMathOperator{\mt}{MT}
\DeclareMathOperator{\smt}{sMT}
\DeclareMathOperator{\U}{U}
\def\iso{\cong}

\def\til{\widetilde}
\def\integ{{\mathbb Z}}
\def\cx{{\mathbb C}}
\def\rat{{\mathbb Q}}
\def\gp{{\mathbb G}}
\def\ff{{\mathbb F}}
\def\ELL{{\mathbb L}}
\renewcommand{\bar}[1]{{\overline{#1}}}

\newcommand{\set}[1]{\left\{#1\right\}}

\allowdisplaybreaks

\title{On the essential torsion finiteness of abelian varieties over torsion fields}
\author{Jeffrey D. Achter}
\address{Colorado State University, Department of Mathematics,
 Fort Collins, CO 80523,
 USA}
\email{j.achter@colostate.edu}

\author{Lian Duan}
\address{ShanghaiTech University, Institute of Mathematical Sciences, No.393 Middle Huaxia Road, Pudong New District, Shanghai 201210, China}
\email{duanlian@shanghaitech.edu.cn}

\author{Xiyuan Wang}
\address{The Ohio State University, Department of Mathematics,
100 Math Tower, 231 W 18th Ave, Columbus, OH 43210, USA}
\email{wang.15476@osu.edu}

\thanks{Research of the first-named author was supported in part by a grant fom the Simons Foundation (637075).}
\begin{document}

\begin{abstract}
    The classical Mordell-Weil theorem implies that an abelian variety $A$ over a number field $K$ has only finitely many $K$-rational torsion points. This finitude of torsion still holds even over the cyclotomic extension $K^{\cyc}=K\Q^{\mathrm{ab}}$ by a result of Ribet. In this article, we consider the finiteness of torsion points of an abelian variety $A$ over the infinite algebraic extension $K_B$ obtained by adjoining the coordinates of all torsion points of an abelian variety $B$. Assuming the Mumford-Tate conjecture, and up to a finite extension of the base field $K$, we give a necessary and sufficient condition for the finiteness of $A(K_B)_{\rm tors}$ in terms of Mumford--Tate groups. 
              We give a complete answer when both abelian varieties have dimension at most three, or when both have complex multiplication.
    \end{abstract}

\maketitle

\tableofcontents

\section{Introduction}

Suppose $A$ is an abelian variety defined over a number field $K$.  The celebrated Mordell-Weil theorem states that for any number field $L$ containing $K$, the subgroup $A(L)_{\tors}$ of torsion points of $A$ defined over $L$ is finite (e.g., \cite[Appendix II]{Mumford-AV}). At the opposite extreme, over the algebraic closure $\overline{K}$ of $K$, using the geometry of $A$ one easily sees that the geometric torsion group $A(\overline{K})_{\tors}$ is infinite. Then it is natural to ask whether the finiteness property of the torsion subgroup $A(L)_{\tors}$ is still preserved for various infinite algebraic extensions $L/K$. This kind of question can be traced back at least to \cite{Maz72}, in which Mazur asked whether the group $A(K^{\cyc,p})$, where $K^{\cyc,p} = K(\cup_n \zeta_{p^n})$ is the field obtained by adjoining all $p$-power roots of unity to $K$, is still finitely generated. The torsion part $A(K^{\cyc,p})_{\rm tors}$ of this group is then proved to be finite by Imai \cite{Imai75} and Serre \cite{Serre-l-adic-repn} independently. Their results are then generalized by Ribet in his article \cite[Appendix]{Katz-Lang-Ribet-finite-thm-geo-CFT}. Let $K^{\rm{cyc}}:= \cup_p K^{\cyc,p} = K(\cup_{n} \zeta_n)$ be the infinite extension of $K$ obtained by adjoining all roots of unity.  Then Ribet showed that for every abelian variety $A$ defined over the number field $K$,  one has 
\begin{equation}\label{Eqn: Rebit_torsion_cyclic_finite}
	|A(K^{\rm{cyc}})_{\rm{tors}}|<\infty.
\end{equation} 
Zarhin then further generalized this result \cite{Zarhin-End-Tor-AV} by showing that if $A$ is a simple abelian variety over its ground field $K$, then over the maximal abelian extension $K^{\rm ab}$ of $K$, the torsion group $A(K^{\rm ab})_{\rm tors}$ is finite if and only if $A$ is not of CM-type over $K$, i.e., if and only if the $K$-endomorphism algebra $\End(A)_\rat := \End(A)\otimes_{\Z}\Q$ is not a number field of degree $2\dim A$. As a cohomological generalization of Ribet's result, R\"{o}ssler and Szamuely \cite{Rossler-Szamuely-torsion-finite} proved that for any projective, smooth, and geometrically connected variety $X$ over a number field $K$, the groups $H^i_{\text{\'et}}(\overline{X}, \rat/\integ (j))^{\Gal(\overline{K}/K^{\rm cyc})}$ are finite for all odd positive integers $i$ and all integers $j$. In contrast, when $K$ is a $p$-adic field, then the analog of Imai and Serre's result is generalized by Ozeki in \cite{Ozeki09}. In addition, an analogue of Zarhin's result is proved for Drinfeld modules by Li \cite{Li-DM-torsion-finite}.  Quite recently, Lombardo studied a problem which, while perhaps superficially different, turns out to be closely related \cite{Lombardo-iso-Kummerian}; we discuss his work at the end of this introduction. 

In this paper, we focus on the generalization of \eqref{Eqn: Rebit_torsion_cyclic_finite} in another direction. Notice that by the Kronecker–Weber theorem, the cyclotomic extension $K^{\cyc} =K({\G}_{m, {\rm tors}})$ is exactly the extension of $K$ obtained by adjoining all the geometric torsion points of the algebraic torus $\G_{m}$. Due to the fact that there is no nontrivial isogeny, or even nonconstant geometric morphism, between $\G_m$ and an abelian variety, one is naturally led to ask the following question: 

\begin{ques}\label{Question: main_question}
Suppose two abelian varieties $A$ and $B$ are defined over a number field $K$; assume that over $\bar K$ they share no common nontrivial isogeny factor. Let $K_{B}$ denote the infinite extension of $K$ obtained by adjoining all the geometric torsion points of $B$. Is the torsion group of $A$ over $K_B$  finite, i.e., is
\[
|A(K_B)_{\rm tors}|<\infty?
\]
\end{ques}

We answer this question in the present paper, up to a finite extension of the base field and under the Mumford--Tate conjecture.  We state our results after introducing a few definitions.

\begin{defn}\label{Defn: ess_tor_fini}
Given two abelian varieties $A$ and $B$ defined over a number field $K$, we say that $A$ is \emph{torsion finite} for $B$ over $K$ if $A(K_B)_{\tors}$ is
finite. Otherwise, we say that $A$ is \emph{torsion infinite} for $B$ over $K$. 

Moreover, if there is a finite extension $L/K$ such that $A(L_B)_{\tors}$ is infinite, we say $A$ is \emph{potentially torsion infinite} for $B$. If such $L$ does not exist, we will say that $A$ is \emph{essentially torsion finite} for $B$. 

\end{defn}

Although not stated in this language, Serre gave a positive answer to Question~\ref{Question: main_question} in \cite[Th\'eorem\`e 6 and 7]{Serre-1972} when $A$ and $B$ are both elliptic curves which are not geometrically isogenous. In fact, Serre proved that for such $A$ and $B$ the image of the adelic representation induced by $A\times B$ equals the product of the images induced by $A$ and $B$, up to a finite index; the claim readily follows. Our strategy is inspired by Serre's work. However, one of the advantages of working with elliptic curves, as opposed to higher dimensional abelian varieties, is the open image theorem \cite{Serre-l-adic-repn,Serre-1972}. This theorem, together with its analogue for CM elliptic curves \cite{Serre-Tate-good-red-AV}, classifies the $\ell$-adic representation images of elliptic curves in terms of  $\rat$-algebraic groups. With the help of these algebraic groups, the answer to Question~\ref{Question: main_question} is essentially (but nontrivially) a consequence of Goursat's lemma. 

 The open image theorem for higher-dimensional abelian varieties is not known in general -- indeed, it cannot hold for an abelian variety which is not Hodge maximal.
 Nonetheless, the Mumford--Tate conjecture claims that, for an abelian variety $A$ over a sufficiently large number field, its $\ell$-adic Galois representation images are still classified by a $\rat$-algebraic group, the Mumford--Tate group $\mt(A)$, which is defined in terms of the Hodge structure  $H_1(A(\C), \rat)$. (For a quick review of the Mumford--Tate group and the related conjecture, see Section~\ref{Sect: MT_conj}.  For an abelian variety $A$ defined over a subfield $K\subset \cx$, we will often abuse notation and write $H_1(A,\rat)$ for the homology group $H_1((A\times_{\spec K} \spec \cx)(\cx)^{\mathrm{an}},\rat)$, endowed with its Hodge structure, and define $H^1(A,\rat)$ in an analogous fashion.)

In this article, assuming the Mumford--Tate conjecture, using Galois theory, and generalizing Serre's idea to algebraic groups beyond $\GL_2$, we are able to prove a criterion for the essential torsion finiteness of pairs of abelian varieties. 
 
\begin{thm}[see Theorem~\ref{T:A abs simp} for a more detailed version]\label{Thm: main_thm_2}Suppose $A$ and $B$ are two  absolutely simple abelian varieties defined over a number field $K$, and suppose that the Mumford--Tate conjecture holds for $A$ and $B$. 
 
Then $A$ is potentially torsion infinite for $B$ if and only if 
\begin{equation}
\label{E:dim MT}
    \dim \mt(A\times B) = \dim \mt(B).
\end{equation}
When this holds, for each prime $\ell$ there exists a finite extension $L_\ell/K$ such that
\[
A[\ell^\infty](L_{\ell, B}) = A_{\ell^\infty}:= A[\ell^\infty](\bar K).
\]
\end{thm}
Insofar as (the image of) the action of Galois on torsion points is constrained by the Mumford--Tate group, it is not surprising that a relation on Mumford--Tate groups can force a resonance among torsion fields.  What is perhaps more interesting is that just the presence of infinite torsion -- for example, if $A[\ell](K_B)$ is nontrivial for an infinite, but sparse, set of primes -- is enough to constrain the relation between $\mt(A\times B)$ and $\mt(B)$.  In particular, we will see that the existence of $\ell$-torsion for infinitely many $\ell$ forces the presence of $\ell$-torsion for $\ell$ in a set of positive density.

The Mumford--Tate group $\mt(A)$ is canonically an extension of $\G_m$ by the Hodge group, or special Mumford--Tate group, $\smt(A)$.  Ichikawa \cite{Ichikawa-Hg-gp-AV} and Lombardo \cite{Lombardo-Hg-gp-pair-AV-2016} have investigated conditions under which 
\begin{equation}
\label{E:smt is product}
\smt(A\times B) = \smt(A) \times \smt(B).
\end{equation}
(For example, this holds if $A$ and $B$ satisfy a certain "odd relative dimension" condition and at least one is  \emph{not} of Type IV in the Albert classification.)  When $A$ and $B$ satisfy \eqref{E:smt is product}, we have
   \[
\dim \mt(A\times B) = \dim(\mt(A))+\dim(\mt(B))-1 > \max\set{\dim \mt(A), \dim \mt(B)}.
\]
Theorem \ref{Thm: main_thm_2} then immediately implies that $A$ and $B$ are mutually essentially torsion finite. (See the last part of Section~\ref{Sect: potentially-infinite} for more details.) 

Taken together with our main theorem, Ichikawa and Lombardo's results often imply a positive answer to our main question \ref{Question: main_question}, except when both $A$ and $B$ are of Type IV. Thus Question \ref{Question: main_question} is particularly interesting 
when both $A$ and $B$ are of Type IV, such as when both $A$ and $B$ have complex multiplication (CM) over $\bar K$.
 
In fact, there do exist examples where Question~\ref{Question: main_question} has a negative answer. For instance, the Jacobian of a certain genus $4$ curve \cite[Example~6.1]{Shioda-av-Fermat-type-1981/82} decomposes into a product of a potentially CM elliptic curve and a simple potentially CM abelian threefold. However, one can check that the elliptic curve is torsion infinite for the threefold (see \cite[Theorem~1.2]{Lombardo-iso-Kummerian}). In addition, Lombardo~\cite{Lombardo-iso-Kummerian} constructed infinitely many pairs of nonisogenous CM abelian varieties for which the answer to Question~\ref{Question: main_question} is again negative.
As a complement to Lombardo's work, we give a sufficient and necessary condition for answering our main question for CM pairs, as follows.

Let $A$ be an isotypic abelian variety over a number field $K$ with CM by a CM field $E$, and suppose that $K$ contains $E$. (Recall that an abelian variety is said to be isotypic if it is isogenous to some power of a simple abelian variety.) We will see in Section \ref{Sect: CM} that there is a surjection of algebraic tori
\[
\xymatrix{
T^K := \Res_{K/\rat}\G_m \ar@{->>}[r] & \mt(A)
}
\]
which induces an inclusion of character groups
\[
\xymatrix{
X^*(\mt(A)) \ar@{^(->}[r] & X^*(T^K).
}
\]
In fact, $\mt(A)$ only depends on the CM type of $A$.
With this reminder, we can state a version of our main theorem for abelian varieties with complex multiplication. For a torus $T$, let $X^*(T)_\rat = X^*(T)\otimes \rat$.

\begin{thm}[see Theorem~\ref{Thm: main_thm_CM_pair} for a more detailed version]\label{Thm: main_thm_3}
    Let $A_1$ and $A_2$ be isotypic potentially CM abelian varieties over a sufficiently large number field $K$, with respective Mumford--Tate groups $T_1$ and $T_2$.  Using the inclusions $X^*(T_i) \hookrightarrow X^*(T^K)$, either:
    \begin{enumerate}[label=(\alph*)]
    \item $X^*(T_1)_\rat \subset X^*(T_2)_\rat$.  Then $A_1$ is potentially torsion infinite for $A_2$.
    
    Moreover, if $X^*(T_1) \subset X^*(T_2)$ and if $A_1$ is simple with nondegenerate CM (\ref{S:nondegenerate}), then $A_1(K_{A_2})_{\tors} = A_1(\bar K)_{\tors}$.
    
    \item $X^*(T_1)_\rat \not\subset X^*(T_2)_\rat$.  Then $A_1$ is essentially torsion finite for $A_2$.
    \end{enumerate}
\end{thm}

Theorem~\ref{Thm: main_thm_3}  is unconditional because the Mumford--Tate conjecture is known for CM abelian varieties (see Lemma~\ref{Lem: MTC_CM_Ab.var}).

We briefly compare this result to Zarhin's work \cite{Zarhin-End-Tor-AV}.
Suppose $B$ has complex multiplication over $K$ but $A$ does not even have potential complex multiplication.  Note that $K_B$ is an abelian extension of $K$.   Zarhin's result implies that $A(K^{\operatorname{ab}})_{\tors}$ is finite; \emph{a fortiori}, $A$ is essentially torsion finite for $B$.  However, if $A$ is also of CM type, then Theorem~\ref{Thm: main_thm_3} gives finer information on whether $A$ is essentially torsion finite for $B$.

A morphism $A \to B$ induces a map of homology groups $H_1(A,\rat) \to H_1(B,\rat)$, and thus a morphism of Tannakian categories $\ang{H_1(A,\rat)} \to \ang{H_1(B,\rat)}$ and ultimately of Mumford--Tate groups $\mt(B) \to \mt(A)$.  More generally, a correspondence between $A^m$ and $B^n$ induces a relation between $\mt(A)$ and $\mt(B)$; and the class of such a correspondence is a Hodge class on $A^m \times B^n$.  

In Section~\ref{Sect: extra cycles}, we will see that if the CM abelian variety $A$ is torsion infinite for the CM abelian variety $B$, then there is a nonempty $\rat$-vector space of interesting Hodge classes on some product $A^m \times B^n$; perhaps not surprisingly, we call such a class a torsion infinite class.  These Hodge classes are "extra", in the sense that they are not in the span of classes pulled back from $A^m$ and $B^n$.  Conversely, we show that the presence of such a class implies that $A$ is torsion infinite for $B$.

Of course, the Hodge conjecture predicts that torsion infinite classes are actually the classes of cycles on $A^m \times B^n$.  It would be interesting to see, even in special cases, if one can geometrically realize torsion infinite classes.

In addition to the above applications, thanks to the work of Moonen and Zarhin on the Hodge groups of abelian varieties of low dimension  \cite{Moonen-Zarhin-Hg-cls-av-low-dim}, one can compare the Mumford--Tate groups of every possible pair of absolutely simple abelian varieties up to dimension $3$. As a consequence, we give a positive answer to Question~\ref{Question: main_question} for most pairs of such abelian varieties. Precisely, following the classification in their paper, we prove:

\begin{thm}[also Theorem~\ref{Thm: torsion_finite_low_dim}]\label{Thm: main_thm_4}

Suppose $A$ and $B$ are absolutely simple abelian varieties over a common number field and assume that they are  non-isogenous over $\C$. Suppose that $\dim A \le \dim B \le 3$.  Then $A$ and $B$ are mutually essentially torsion finite except for the following cases:
\begin{enumerate}[label=(\alph*)]
\item $A$ is a CM elliptic curve and $B$ is a CM abelian threefold. 
Then $B$ is  essentially torsion finite for $A$; and $A$ is potentially torsion infinite for $B$ exactly when there is an embedding of $\rat$-algebras $\End^0(A)\hookrightarrow \End^0(B)$. 

    \item  $A$ is a CM elliptic curve and $B$ is an abelian threefold of type IV but not CM.  Then  $B$ is  essentially torsion finite for $A$; and $A$ is potentially torsion infinite for $B$ exactly when there is an isomorphism of $\rat$-algebras $\End^0(A)\iso  \End^0(B)$.
    \item $A$ and $B$ are both CM abelian threefolds. 
\end{enumerate}
\end{thm}
\noindent (In (c), the essential torsion finiteness depends on the CM-types of $A$ and $B$ as in Theorem~\ref{Thm: main_thm_CM_pair}.)

Again, this result is unconditional since the Mumford--Tate conjecture is known to hold  for simple abelian varieties of dimensions less than $4$ \cite{Moonen-Zarhin-Hg-cls-av-low-dim}.

 This paper is structured as follows.  In Section~\ref{Sect: preliminaries}, we collect some basic results on representations of algebraic groups.  In particular, we introduce the notion of a collection of subgroups of bounded index (of the $\ff_\ell$-points of a group scheme over $\integ[1/N]$); this allows us to infer information about a representation of an algebraic group from data about the behavior of abstract subgroups of its finite-field-valued points.  In Section \ref{Sect: tor_pts_on_ab.vars}, we establish notation and review facts (and conjectures)  concerning the Galois representations attached to abelian varieties.  We finally turn to the torsion-finiteness question itself in Section~\ref{SS:MT pair}, establishing our main result (Theorem~\ref{Thm: main_thm_2}) in Section~\ref{Sect: potentially-infinite}.  The paper concludes with a detailed analysis of CM (\S \ref{Sect: CM}) and low-dimensional (\S \ref{Sect: low dim}) pairs of abelian varieties, and of certain extra Hodge classes which are the hallmark of torsion-infinite pairs of CM abelian varieties (\S \ref{Sect: extra cycles}).

 It turns out that while we were working out these results, Lombardo studied a similar problem with somewhat stronger restrictions \cite{Lombardo-iso-Kummerian}.  Two abelian varieties $A$ and $B$ over a number field $K$ are said to be \emph{strongly iso-Kummerian} if for each positive integer $d$ we have 
\begin{equation}\label{Eqn: iso-Kummerian}
K_{A,d} = K_{B,d}
 \end{equation}
i.e., if the $d$-torsion points of $A$ and $B$ generate the same extension of $K$. Using the theory of the (special) Mumford--Tate group and assuming the Mumford--Tate conjecture, Lombardo proves that condition \eqref{Eqn: iso-Kummerian} puts a strong restriction on the Hodge groups of $A$, $B$ and $A\times B$.  This constraint forces $A$ to have the same isogeny factors as $B$ when either $\dim A\leq 3$ and $\dim B\leq 3$ \cite[Theorem~1.2]{Lombardo-iso-Kummerian}; or every simple factor of $A$ or $B$ has dimension $\leq 2$, is of odd relative dimension and not of type IV \cite[Theorem~1.4]{Lombardo-iso-Kummerian}. As a complement, by studying certain simple CM types on cyclic CM fields, Lombardo also constructs infinitely many non-isogenous iso-Kummerian pairs \cite[Theorem~1.1]{Lombardo-iso-Kummerian}. In spite of the obvious similarities, our work is differs from Lombardo's in its emphasis and results.
\begin{enumerate}
    \item Condition~\eqref{Eqn: iso-Kummerian} is much stronger than our (potentially) torsion infinite condition. In fact, \eqref{Eqn: iso-Kummerian} forces $K_A=K_B$, so $K_AK_B/K_B$ is a trivial extension. However, even if $A$ is torsion-infinite for $B$,  $K_AK_B/K_B$ can still be infinite  (Example~\ref{Exa: torsion_infinite-11}).
    
    \item In assumption~\eqref{Eqn: iso-Kummerian}, by taking $d=\ell^n$ and letting $n\to \infty$, one can directly deduce $A[\ell^{\infty}](K_{B,\ell^{\infty}})=A_{\ell^\infty}$. However, if one only assumes that $A$ is torsion-infinite for $B$, it is possible that the subgroup $A[\ell^{\infty}](K_{B,\ell^{\infty}})$ is finite for every $\ell$, but nontrivial for infinitely many $\ell$. One of our main contributions in this paper is to rule out this possibility (under the Mumford--Tate conjecture, as usual). 

    \item Lombardo shows that if $A$ and $B$ are strongly iso-Kummerian, then the natural projections $\mt(A\times B)\to \mt(A)$ and $\mt(A\times B) \to \mt(B)$ are isogenies \cite[Lemma~3.2]{Lombardo-iso-Kummerian}.  We are able to deduce this conclusion from the weaker hypothesis that $A$ and $B$ are mutually potentially torsion infinite (Corollary~\ref{C:mutually torsion infinite}).

\end{enumerate}

\subsection*{Acknowledgments} We thank Yuan Ren for bringing this interesting question to our attention. The second author thanks Ken Ribet for useful suggestions; 
the third author thanks Stefan Patrikis for many enlightening discussions; and we all thank Davide Lombardo for helpful comments on a draft of this paper.

This paper benefited greatly from a referee's extraordinarily close reading; we're grateful for their efforts.

\section{Preliminaries}\label{Sect: preliminaries}

\subsection{Reminders on algebraic groups}\label{Sect: Goursat's_lemma}

We collect some standard, useful facts on algebraic groups.

\begin{lemma}[Goursat's Lemma]\label{L:goursat}
  Let $G_1$, $G_2$, and $G_{12}$ be either abstract groups or algebraic groups over a field.  Suppose $G_{12}$ is endowed with an inclusion $\iota\colon G_{12}\hookrightarrow G_1\times G_2$ such that $\pi_i \circ \iota$ is surjective for $i=1,2$:
\begin{equation}\label{D:goursat}
\xymatrix{
G_{12} \ar@{^(->}[r]^\iota & G_1 \times G_2 \ar@{->>}[dl]^{\pi_1} \ar@{->>}[dr]_{\pi_2} \\
G_1 && G_2
}.
\end{equation}
Let $M_{12} = \ker(\pi_2\circ\iota)$, and let $H_{12} \iso M_{12}$ be
the image of $M_{12}$ under the isomorphism $G_1 \times \set e \iso G_1 $; define $M_{21}$ and $H_{21}$ analogously.

Then under the composite map
\[
  \xymatrix{
    G_{12} \ar@{^(->}[r] & G_1 \times G_2 \ar[r] &
    \frac{G_1}{H_{12}}\times \frac{G_2}{H_{21}},
  }
\]
$G_{12}$ is the inverse image in $G_1 \times G_2$ of the graph of an isomorphism $\frac{G_1}{H_{12}} \to
\frac{G_2}{H_{21}}$.
\end{lemma}

\begin{proof}
This is standard; see, e.g., \cite[Lemma~5.2.1]{Ribet-Gal-repn-AV-RM} for the case of abstract
groups.  The constructions of $H_{ij}$ and $M_{ij}$ also make sense in the category of algebraic groups, and the asserted properties may be verified pointwise, as in \cite{Ribet-Gal-repn-AV-RM}.
\end{proof}

\begin{remark}
  \label{R:goursat product}
Under the hypotheses of Lemma \ref{L:goursat}, suppose $H_{12} =
G_{1}$.  Then clearly $H_{21} = G_2$, and thus $G_{12} \cong G_1 \times
G_2$.  At the opposite extreme, if $H_{12}$ and $H_{21}$ are trivial, then, up to a choice of isomorphism $G_1 \iso G_2$, $\iota: G_{12}\hookrightarrow G_1\times G_2$ is the diagonal embedding.
\end{remark}

\begin{lemma}\label{Lem: dim_and_rank}
Assume $G_1$, $G_2$ and $G_{12}$ are reductive groups over a
field $K$ of characteristic zero and satisfy a diagram \eqref{D:goursat}. Then the following are equivalent:
\begin{enumerate}[label=(\alph*)]
    \item $\dim G_{12}=\dim G_2$; 
    \item $\rank G_{12}=\rank G_2$; and
\item the surjection $G_{12} \twoheadrightarrow G_2$ is an isogeny.
\end{enumerate}
\end{lemma}
\begin{proof}
The short exact sequence of algebraic
groups 
\[\xymatrix{
0 \ar[r] &M_{12} \ar[r]&  G_{12}\ar[r]^{\pi_2\circ \iota}&  G_2 \ar[r]& 0}
\]
induces a corresponding exact sequence on Lie algebras. Since
$K$ has characteristic zero, the rank and dimension of a reductive
group can be read off from its Lie algebra.  Note that $\operatorname{Lie}(M_{12}) = \operatorname{Lie}(M_{12}^\circ)$ and that $M_{12}^\circ$, being a connected normal (Lemma \ref{Lem: component is normal}) subgroup of a reductive group, is also reductive (e.g., \cite[Cor.~21.53]{MilAlg-gp}).  In particular, either (a)
or (b) holds if and only if $\operatorname{Lie}(M_{12})= (0)$, i.e.,
$\dim M_{12}=0$ and thus $\pi_2\circ
\iota$ is an isogeny.
\end{proof}

\begin{lemma}
\label{Lem: component is normal}
  Let $G$ be a connected algebraic group, and let $M\subset G$ be a normal
  algebraic subgroup.  Then $M^\circ$ is normal in $G$.
\end{lemma}

\begin{proof}
  Since $G$ and $M^\circ$ are connected, the image of $M^\circ$ under
  conjugation by $G$ is connected and contains the identity element of $G$.  Since this image is a subgroup of $M$, which is normal in $G$, it
  is contained in $M^\circ$, and thus $M^\circ$ is stable under
  conjugation by $G$.
\end{proof}

Finally, when studying CM abelian varieties in Section \ref{Sect: application}, we will need to work with algebraic tori.

Let $K$ be a perfect field.  An algebraic torus $T/K$ is an algebraic group such that $T_{\bar K} \iso \mathbb G_{m,\bar K}^{\oplus \dim T}$.  Let $X^*(T)$ be the (absolute) character group $X^*(T) = \Hom(T_{\bar K}, \mathbb G_{m,\bar K})$, and let $X^*(T)_\rat  = X^*(T)\otimes\rat$; then $T\mapsto X^*(T)$ gives a contravariant equivalence between the category of algebraic tori over $K$  and the category of finite free $\integ$-modules with a continuous action by the absolute Galois group $\gal(K) := \gal(\bar K/K)$.  This extends to a contravariant equivalence between the category of $K$-groups of multiplicative type and the category of finitely generated $\integ$-modules with continuous $\gal(K)$ action.  We have $\dim T = \rank_{\integ}X^*(T) = \dim_{\rat} X^*(T)_\rat$.  If $\alpha\colon S \to T$ is a morphism of algebraic tori, then $\dim \ker(\alpha) = \dim_\rat X^*(T)_\rat/ \alpha^*X^*(S)_\rat$, and $\alpha$ has connected kernel if and only if $X^*(T)/\alpha^*X^*(S)$ is torsion-free.

If $F/\rat$ is a finite extension, we let $T^F$ denote $\operatorname{Res}_{F/\rat}\gp_{m,F}$, the Weil restriction of the multiplicative group, and let $T^{F,1}$ denote the norm one torus $\operatorname{Res}_{F/\rat}^{(1)}\gp_{m,F}$, which is the kernel of the norm map $\operatorname{N}_{F/\rat}\colon T^F \to \gp_m$.

\subsection{Representations of algebraic groups}\label{Sect: repn_alg_gp}

\subsubsection{Fixed spaces}
\label{S:fixed spaces}

Let $G/K$ be an algebraic group over a field.  Let $V/K$ be a
finite-dimensional representation of $G$, i.e., a finite-dimensional
vector space $V$ equipped with a morphism $G \to \GL_V$ of algebraic groups. The schematic fixed space of $V$ under $G$ is 
\[
V^G = \set{v\in V: g\cdot v_R=v_R \ (\text{in }V_R) \text{ for all $K$-algebras $R$ and all $g\in {G}(R)$}},
\]
where $v_R$ is the image of $v$ under $V\hookrightarrow V\otimes_KR$ \cite[\S 4i]{MilAlg-gp}.

We define the na\"ive fixed space as
\[
V^{G(K)} = \set{ v \in V: g\cdot v = v\text{ for all } g\in G(K)}.
\]
More generally, if $\Gamma\subset G(K)$ is an abstract subgroup, the subspace fixed by $\Gamma$ is
\[
  V^\Gamma = \set{ v \in V: g\cdot v = v\text{ for all }g \in \Gamma}.
\]

\begin{lemma}
We have $V^G \subseteq V^{G(K)}$, with equality if $K$ is infinite and
$G/G^\circ$ is a split \'etale group.
\end{lemma}

\begin{proof}
The first statement is trivial; for the second, use the fact that
under the stated hypotheses, $G(K)$ is Zariski dense in $G$.
\end{proof}

\begin{lemma}
\label{Lem: fixed_subsp_sub_repn}
Let $\rho\colon G \to \GL_V$ be a morphism of algebraic groups over $K$,
and let $M\subset G$ be a normal algebraic subgroup.  Then $V^M$ is
stable under $G$, and thus is a sub-$G$-representation of $G$.
\end{lemma}

\begin{proof}
It suffices to verify this after passage to the algebraic closure of
$K$, so we may and do assume that $G(K)$ is dense in $G$, and that
$M(K)$ is dense in $M$. It now suffices to show that, for each $g\in
G(K)$, $gW \subset W$.  Since $W$ is fixed by the normal subgroup $M$, $gW$ is fixed by
$gMg\inv = M$, and so $gW \subset V^M = W$.
\end{proof}

\begin{lemma}
\label{L:fixed by index ell subgroup}
    Let $V/\ff_\ell$ be a finite-dimensional vector space, and let $\mathtt G$ be an abstract group equipped with a representation $\rho: \mathtt G \to \GL_V(\ff_\ell)$.  Let $\mathtt H \trianglelefteq \mathtt G$ be a normal subgroup such that $V^{\rho(\mathtt H)}\supsetneq (0)$ and $[\mathtt G: \mathtt H]$ is a power of $\ell$.  Then $V^{\rho(\mathtt G)}\supsetneq (0)$.
\end{lemma}

\begin{proof}
    Choose some $w \in V^{\rho(\mathtt H_\ell)}\smallsetminus \set 0$, and let $W = \ff_\ell[\mathtt G]w$ be the subspace spanned by its $\mathtt G$-orbit.  The representation $\rho_W$ of $\mathtt G$ on $W\subseteq V$ factors as $\mathtt G \to \mathtt G/\mathtt H \to \GL_W(\ff_\ell)$.  Since $\mathtt G/\mathtt H$ is an $\ell$-group, by the Sylow theorem, $\rho_W(\mathtt G)$ is contained in the $\ff_\ell$-points of a maximal unipotent subgroup $U$ of $\GL_W$.  (Differently put, after a suitable choice of basis, the image of $\mathtt G$ in $\GL_W$ is contained in the $\ff_\ell$-points of the group $U$ of unipotent upper-triangular matrices.)  Since $U$ has a nontrivial fixed vector, so does $\mathtt G$.
\end{proof}

 \subsubsection{Bounded subgroups}\label{Sect: bounded_subgps}
Let $H/\integ[1/N]$ be a smooth affine algebraic group scheme with geometrically connected fibers.  Suppose that for each
$\ell\nmid N$, an abstract group $\mathtt H_\ell \subset H(\ff_\ell)$
is specified.

\begin{defn}
  \label{D:bounded}
  The collection
$\set{\mathtt H_\ell}_\ell$ is bounded (in $H$, independently of
$\ell$) if there exists some finite $B$ such that, for each $\ell$,
\begin{equation}
  \label{E:defbounded}
[ H(\ff_\ell):\mathtt H_\ell]< B.
\end{equation}
Equivalently,  there is some positive $C = 1/B$ such that
$\#\mathtt H_\ell  > C \#H(\ff_\ell)$.
\end{defn}

\begin{lemma}\label{L:ffindexgroups}
  Let $\alpha\colon H \to G$ be a surjective morphism of smooth algebraic
  groups over $\integ[1/N]$, with $G$ connected.  Then $\set{\alpha(H(\ff_\ell))}_\ell$ is
  bounded in $G$.
\end{lemma}

\begin{proof}
Let $P = \ker(\alpha)$.  Its formation commutes with base change, and it is the extension of a finite group scheme $D$ 
by a connected group $P^\circ$.  Taking $\ff_\ell$-points, we have
\[
\xymatrix{
1 \ar[r] & P(\ff_\ell) \ar[r] & H(\ff_\ell) \ar[r]^{\alpha(\ff_\ell)} & G(\ff_\ell) \ar[r] & H^1(\ff_\ell, P).
}
\]
It suffices to show that $H^1(\ff_\ell,P)$ is bounded independently of
$\ell$.   By Lang's theorem $H^1(\ff_\ell, P^\circ)$ is trivial, and so
it suffices to show that $H^1(\ff_\ell, D)$ is bounded. Now use the fact that $D$ is finite and $\#H^1(\ff_\ell,D) = \#D(\ff_\ell)$ (\cite[p.~290]{PlatonovRapinchuk}).
\end{proof}

\begin{lemma}
\label{L:abstractboundedindex}
Let
\[
\xymatrix{ 0 \ar[r] & P \ar[r] & H \ar[r]^\alpha & G \ar[r] & 0}
\]
be an exact sequence of algebraic groups over $\integ[1/N]$.  Suppose that $\set{\mathtt H_\ell}_\ell$ has bounded index in $H$.  Then
\begin{enumerate}[label=(\alph*)]
\item $\set{\alpha(\mathtt H_\ell)}_\ell$ has bounded index in $G$, and
\item $\set{ (\ker \alpha|_{\mathtt H_\ell})\cap P^\circ(\ff_\ell)}_\ell$ has bounded index in $P^\circ$.
\end{enumerate}
\end{lemma}

\begin{proof}
Suppose $\#\mathtt H_\ell> C_{\mathtt H}\cdot  \#H(\ff_\ell)$ and (using Lemma \ref{L:ffindexgroups}) $\#\alpha(H(\ff_\ell)) > C_\alpha \cdot \#G(\ff_\ell)$ for $\ell \gg 0$. We then have the easy estimates
\[
    \#\alpha(\mathtt H_\ell) = \frac{\#\mathtt H_\ell}{\#\ker \alpha|_{\mathtt H_\ell}} \ge \frac{\# \mathtt H_\ell}{\#P(\ff_\ell)}> C_{\mathtt H} \frac{\#H(\ff_\ell)}{\#P(\ff_\ell)} = C_{\mathtt H} \cdot  \# \alpha(H(\ff_\ell)) {>} C_{\mathtt H}C_\alpha \cdot  \#G(\ff_\ell),
\]
This proves (a).

Let $\mathtt P_\ell = \ker \alpha|_{\mathtt H_\ell}$.  Then
\[
\#\mathtt P_\ell = \frac{\#\mathtt H_\ell}{\# \alpha(\mathtt H_\ell)} \ge \frac{\#\mathtt H_\ell}{\#\alpha(H(\ff_\ell))} 
> C_{\mathtt H}\frac{\#H(\ff_\ell)}{\#\alpha(H(\ff_\ell))} = C_{\mathtt H} \#P(\ff_\ell).
\]
Let $P'$ be any irreducible component of $P_{\ff_\ell}$.  Then $\mathtt P_\ell\cap P'(\ff_\ell)$, if nonempty, is a torsor under $\mathtt P_\ell \cap P^\circ(\ff_\ell)$, and so 
\[
\# (\mathtt P_\ell \cap P^\circ(\ff_\ell)) \ge \frac 1{[P:P^\circ]}\#\mathtt P_\ell.
\]
\end{proof}

 \subsubsection{Representations of connected groups}

Let $G/\integ[1/N]$ be a smooth affine algebraic group with connected fibers.  Let $V$
be a free $\integ[1/N]$-module of rank $n$, and let $\rho\colon G \to \GL_V$ be a
representation.  For a field $k$ equipped with a ring map
$\integ[1/N]\to k$, let $r_k(G,\rho)$ be the multiplicity of the trivial
representation of $G(k)$:
\begin{equation}\label{Eqn: defn_r_k}
    r_k(G,\rho) = \dim_k (V\otimes k)^{G(k)}.   \end{equation}
  
Let $V_\ell = V\otimes \ff_\ell$, and let $r_\ell(G,\rho) = r_{\ff_\ell}(G,\rho)$.  Note that, when $\ell \gg 0$, by
specialization, we always have $r_\rat(G,\rho) \leq r_\ell(G,\rho)$.

If $g\in G(k)$, let $m(g,\rho) = m_k(g,\rho)$ be the multiplicity of
$1$ as a root of the characteristic polynomial of $\rho(g)$.  Let
$G_{\rho, \ge m}$ be the locus of those $g$ for which $m(g,\rho) \ge
m$.  (Schematically, $G_{\rho, \ge 1}$ may be constructed by pulling
back the composite morphism
\[
  \xymatrix{
    G \ar[r]^\rho & \GL_V \ar[r]^{\text{charpoly}} & \mathbb
    G_a^n \ar[r]^{\text{eval}_1} & \mathbb G_a
  }
\]
by the zero section $\spec \integ[1/N] \to \mathbb G_a$, where the map $\text{eval}_1$ means evaluating the characteristic polynomial at $1$; for other
values of $m$, $G_{\rho,\ge m}$ may be constructed by considering
higher derivatives of the characteristic polynomial.)

\begin{lemma}\label{L:r_ell_implies_r_Q}
Suppose that there is an infinite collection of primes $\ELL$ such that, if $\ell \in \ELL$, then $r_\ell(G,\rho) = r$. Then we have:
\begin{enumerate}[label=(\alph*)]
\item For each $g\in G(\rat)$, $m(g,\rho) \ge r$;
\item $r_\rat(G, \rho) = r$; and
\item $\dim V_\rat^{G_\rat} = r$.
\end{enumerate}
\end{lemma}

\begin{proof}
We assume $r>0$, since (by specialization) the statement is trivial if $r=0$.

For (a) it suffices to apply, to the characteristic polynomial of
$\rho(g)$, the following elementary observation. Let
$f(T) \in \rat[T]$ be any polynomial; since ``clearing denominators''
does not alter the roots of $f$, we may and do assume
$f(T) \in \integ[T]$.  Suppose $\lambda \in \integ$. If $\ell$ is
sufficiently large, relative to the coefficients of $f$ and to $\lambda$, then
$f(\lambda) = 0$ if and only if $f(\lambda) \equiv 0 \bmod \ell$; and,
by taking the first $r-1$ derivatives of $f$, a similar result holds
for roots of higher multiplicity.

We now prove (b). 
For each $\ell \in \ELL$, let $Y_\ell \subset V_\ell$ be
the subspace fixed by $G_{\ff_\ell}$.

Let $m_0$ be the integer such that $G_\rat= {G_\rat}_{\rho, \ge m_0}
\supsetneq {G_\rat}_{\rho, \ge m_0+1}$; by (a), we have $m_0 \ge r \ge
1$.  Let ${G_\rat}^{ss}$ be the open and dense semisimple locus (e.g., \cite[Theorem~22.2]{Humphreys-Linear-Alg-gp}), and let $G^* = {G_\rat}^{ss} \smallsetminus
{G_{\rat}}_{\rho, \ge m_0+1}$.  Since $G_\rat$ is connected,
$G^*$ is open and dense in $G_\rat$.  Like any connected affine algebraic group, $G_\rat$ is
unirational. Consequently, $G^*(\rat)$ is Zariski dense in
$G_\rat$.

For $g\in G^*(\rat)$, let $W_g \subset V_\rat$ be the
$m_0$-dimensional subspace fixed by $g$.  After choosing an integral
model of $W_g$, for all but finitely many $\ell$ the reductions
$g_\ell\in G(\ff_\ell)$ and $W_{g,\ell} \subset
V_\ell$ are well-defined; and for $\ell \in \ELL$, we have
$W_{g,\ell} \supseteq Y_\ell$.

Let $W = \cap_{g\in G^*(\rat)} W_g$; since $V_\rat$ is
finite-dimensional, there is a finite list of elements $g_1, \dots,
g_n \in G^*(\rat)$ such that $W = \cap_i W_{g_i}$. If  $\ell$
is sufficiently large as to avoid the finitely many primes of bad
reduction for the $W_{g_i}$, then   $W\otimes \ff_\ell$ contains
$Y_\ell$.  This shows that $\dim_\rat W \ge \dim_{\ff_\ell} Y_\ell =
r$.  By the density of $G^*(\rat)$, $W$ is fixed by all of
$G_\rat$, and so $r_\rat(G,\rho) \ge r$; again, by
specialization, we find that equality holds. This proves (b).  Since $W = V_\rat^{G_\rat}$, we may conclude (c), as well.
\end{proof}

Now let $\set{\mathtt G_\ell}$ be a collection of bounded subgroups of
$G$, and let
\[
  r_\ell(\mathtt G_\ell,\rho) = \dim{V_\ell}^{\mathtt G_\ell}.
\]

In the statement below, "sufficiently large" depends only on $\dim V$ and the constant in  \eqref{E:defbounded}; however, in our applications, we don't have control over this constant.
\begin{lemma}
  \label{L:specialfix}
Let $\set{\mathtt G_\ell}$ be a collection of bounded subgroups of
$G$.  If $r_\ell(\mathtt
G_\ell,\rho) = r$ for some sufficiently large $\ell$, then $r_\ell({G},\rho) = r$.
\end{lemma}

\begin{proof}
For any $\ell$, let $W_\ell \subset V\otimes \ff_\ell$ be
a subspace of dimension $r$, and let $\fix_{G,W_\ell}\subset
G_{\ff_\ell}$ be the subgroup scheme which fixes $W_\ell$.  By
B\'ezout's theorem, since
$\fix_{G,W_\ell}$ is the intersection of $G$ and $\fix_{\GL_{V,\ff_\ell},W_\ell}$ in $\GL_{V,\ff_\ell}$, there is a constant
$B$ such that $\#\pi_0(\fix_{G,W_\ell}) \le B$, independent of the
choice of $W_\ell$ and of $\ell$. 

For any connected group $H$ of dimension $d$ over a finite field
$\ff$, we have $(\#\ff-1)^d \le \#H(\ff) \le (\#
\ff+1)^d$ (\cite[Lemma~3.5]{Nori-subgp-GLn(Fp)} or \cite[Proposition~3.1]{Larsen-Pink-finite-gp-of-alg-gp}).  Fix a prime $\ell$, and suppose that $r_\ell(\mathtt
G_\ell,\rho) = r$; let $W_\ell \subset V\otimes \ff_\ell$ be the
subspace fixed by $\mathtt G_\ell$.  We then have
\[
\#\fix_{G,W_\ell}^\circ(\ff_\ell) \ge \frac{1}{B}
\#\fix_{G,W_\ell}(\ff_\ell) \ge \frac{1}{B} \#\mathtt G_\ell \ge \frac{C}{B} \#G(\ff_\ell) \ge \frac{C}{B} (\ell-1)^d.
\]
If $\ell\gg_{d,C/B} 0$, this forces $\dim \fix_{G,W_\ell} = \dim G_{\ff_\ell}$, so
that $\fix_{G,W_\ell} =G_{\ff_\ell}$.
\end{proof}

\begin{lemma}
\label{L:rlGl}
  Let $\set{\mathtt G_\ell}$ be a collection of bounded subgroups of
  $G$.  Suppose that there is an infinite collection of primes $\ELL$
  such that, if $\ell \in \ELL$, then $r_\ell(\mathtt G_\ell,\rho) \ge
  r$. Then
\begin{enumerate}[label=(\alph*)]
\item $r_{\rat}(G,\rho)\ge r$ and
\item $r_\ell(G,\rho) \ge r$ for all but finitely many $\ell$.
\end{enumerate}
\end{lemma}

\begin{proof}
  By Lemma \ref{L:specialfix}, we find
  that for $\ell \gg0$, we have $r_\ell({G}_\ell,\rho) \ge r$.
  Lemma \ref{L:r_ell_implies_r_Q} then shows that
  $r_\rat(G,\rho) \ge r$.  This proves (a); then (b) follows by
  specialization.
  \end{proof}

\begin{lemma}
  \label{L:rZlGZl}
  Let $\mathtt G_{\ell_0^\infty}$ be a Zariski dense subgroup of
  $G_{\rat_{\ell_0}}$.   Suppose
  $r(\mathtt G_{\ell_0^\infty},\rho) \ge r$.  Then
\begin{enumerate}[label=(\alph*)]
\item $r_{\rat}(G,\rho)\ge r$, and 
\item $r_\ell(G,\rho) \ge r$ for all but finitely many $\ell$.
\end{enumerate}
\end{lemma}

\begin{proof}
Under the hypothesis, $r(G_{\rat_{\ell_0}}, \rho) \geq r$; for (a), it then suffices to note
that, since $G_\rat$ is connected, the formation of the fixed points of the action of $G$ is stable
under the base change $\rat \hookrightarrow \rat_{\ell_0}$ (\S
\ref{S:fixed spaces}).  Part (b) follows by specialization.
\end{proof}

\subsubsection{Interlude on \'etale group schemes}\label{Sect: etale_gp_sch}

If $(S,\bar s)$ is a geometrically pointed connected scheme, then a (not necessarily connected) finite \'etale group scheme $G \to S$ is tantamount to an action, by group automorphisms, of $\pi_1(S,\bar s)$ on the abstract finite group $G_{\bar s}$.  We will say that $G$ is split if this action is trivial.

\begin{lemma}
  \label{L:etaleoftentrivial}
  Let $G/\integ[1/N]$ be an \'etale group scheme, and let $M\subset G$
  be a normal sub-group scheme.  Suppose that there exists an $\ell_0
  \nmid N$ such that $M(\ff_{\ell_0}) = G(\ff_{\ell_0})$.  Then for
  $\ell$ in a set of positive density, $M(\ff_\ell) = G(\ff_\ell)$; and if $G$ is split, then $M(\ff_\ell) = G(\ff_\ell)$ for every $\ell\nmid N$.
\end{lemma}

\begin{proof}
  Let $\spec(R) \to \spec(\integ[1/N])$ be a Galois \'etale cover which
  trivializes $G$ and $M$; let $K = \operatorname{Frac}(R)$.  Let $\ell\nmid N$ be a prime, and let $\lambda$ be
  a prime of $R$ lying over $\ell$.  The Artin symbol $(\lambda,
  K/\rat)$ determines $G(\ff_\ell)$ and $M(\ff_\ell)$ as abstract
  groups, and the equality $M(\ff_\ell) = G(\ff_\ell)$ depends only on
  the conjugacy class $(\ell, K/\rat)$.

  Under the hypotheses, for any $\ell$ in the set of positive density
  for which $(\ell, K/\rat)= (\ell_0,K/\rat)$, we have $M(\ff_\ell) =
  G(\ff_\ell)$.

  The claim when $G$ is split is trivially true, since then $M$ is split, too, and $M(\ff_\ell) = M(\integ[1/N])$ and $G(\ff_\ell) = G(\integ[1/N])$.
\end{proof}

 \subsubsection{Representations of group schemes}
 We now turn to working with a smooth group scheme $G/\integ[1/N]$.  We will often assume that $G$ has reductive connected component of identity, i.e., that for each $s \in \spec \integ[1/N]$, $(G_s)^\circ$ is reductive.  With this hypothesis, $G^\circ$ is a reductive group scheme over $\integ[1/N]$, and $G/G^\circ$ is \'etale \cite[Prop.~3.1.3]{Conrad-reductive}.  (In fact, we will use the same nomenclature, and deduce the same conclusions, for group schemes over an arbitrary base.) (Without this assumption, it is known that $G/G^\circ$ is an \'etale algebraic space \cite[Lemma 2.1]{Ancone-Huber-Pepin_Lehalleur-2016}, but one may need to enlarge $N$ to ensure that $G/G^\circ$ is representable by a scheme \cite[Prop.~5.1.1]{Behrend-ell-adic-cat-alg-stacks}.)
\begin{lemma}
  \label{L:infmodell}
  Let $G/\integ[1/N]$ be a smooth affine algebraic group scheme with reductive connected component of identity.  Let $V$ be
a free $\integ[1/N]$-module of finite rank, and let $\rho\colon G \to \GL_V$
be a homomorphism of algebraic groups.  Let $\set{\mathtt G_\ell}$ be a collection of bounded subgroups of $G$.  Suppose that there is an infinite collection of primes $\ELL$ such that, if $\ell \in \ELL$, then $r_\ell(\mathtt G_\ell, \rho) \ge r$.
\begin{enumerate}[label=(\alph*)]
\item Then $r_\rat(G^\circ,\rho) \ge r$.
\item Suppose that for some $\ell_0 \in \ELL$, $\mathtt G_{\ell_0}$ meets every geometrically irreducible component of $G_{\ell_0}$.  Then for $\ell$ in a set of positive density, $r_\ell(G,\rho) \ge r$.
\item In the setting of (b), suppose that $G/G^\circ$ is split.  Then $r_\ell(G,\rho) \ge r$ for all $\ell\nmid N$.
\end{enumerate}
\end{lemma}
\begin{proof}
By Lemma \ref{L:rlGl}, applied to the bounded subgroups ${\mathtt
  G_\ell \cap G^\circ(\ff_\ell)}$ of $G^\circ$ we find that
$r(G^\circ_\rat,\rho) \ge r$; this proves (a).

Now suppose that there is some $\ell_0\in \ELL$ for which $\mathtt G_{\ell_0}$ meets every geometrically irreducible component of $G_{\ell_0}$.  Let $\bar G = G/G^\circ$. The image of $\mathtt G_{\ell_0}$ in $\bar G(\ff_{\ell_0})$ is all of $\bar G(\ff_{\ell_0})$.
Replace $V$ with the eigenspace where $G^\circ$ acts with eigenvalue
one; then $V$ has rank at least $r$.  The representation $\rho\colon  G \to
\GL_V$ factors through $\bar\rho\colon  \bar G \to \GL_W$.  There is some subrepresentation $\bar\rho_W\colon  \bar G \to \GL_W \subset \GL_V$ of rank at least $r$ such that $\mathtt G_{\ell_0}$ acts trivially on $W\otimes \ff_{\ell_0}$.

Let $M$ be the group scheme $M   = \ker(\bar\rho_W) \subset \bar G$.  
We have $\mathtt G_{\ell_0} \subseteq M(\ff_{\ell_0})$.
Parts (b) and (c) now follow from Lemma \ref{L:etaleoftentrivial}.
\end{proof}

\begin{lemma}
  \label{L:GcircandGmodell}
  Let $G/\integ_\ell$ be a smooth affine group scheme with reductive connected component of identity, and suppose
  that $\ell \nmid [G:G^\circ]$.  Let $V$ be a
  free $\integ_\ell$-module of finite rank, and let $\rho\colon G \to \GL_V$
  be a homomorphism of algebraic groups.  Suppose that
  $\dim V_{\rat_\ell}^{G^\circ_{\rat_\ell}} \ge r$
     and $r_\ell(G,\rho) \ge r$.
  Then $r(G_{\rat_\ell},\rho) \ge r$.
\end{lemma}

\begin{proof}
Replacing $V$ with the subspace fixed by $G^\circ(\rat_\ell)$, we
assume that the representation $G \to \GL_V$ factors through $\bar G
:= (G/G^\circ)$. Since $\bar G$ is \'etale, specialization of sections gives
a bijection $\bar G(\integ_\ell) \to \bar G(\integ/\ell)$.  Moreover,
by Lang's theorem, $\bar G(\integ/\ell) =
G(\integ/\ell)/G^\circ(\integ/\ell)$ and $\bar G(\integ_\ell) =
G(\integ_\ell)/G^\circ(\integ_\ell)$. 

Since $\ell \nmid [G:G^\circ]$, we may write $V$
uniquely as a direct sum of irreducible $\bar G$ representations over
$\integ_\ell$.  By hypothesis, there is a free $\integ_\ell$-module
$W\subset V$, stable under $G$ and of rank at least $r$, such that $G(\ff_\ell)$ acts trivially
on $W\otimes \ff_\ell$.  We have a commutative diagram
\[
  \xymatrix{
    G(\integ_\ell) \ar[r] \ar[d] & G(\integ_\ell)/G^\circ(\integ_\ell)
    \ar[r] \ar[d]^\sim
    & \GL_W(\integ_\ell) \ar[d] \\
    G(\integ/\ell) \ar[r] & G(\integ/\ell)/G^\circ(\integ/\ell) \ar[r]
    & \GL_W(\integ/\ell)\\
      g \ar@{|->}[r] & \bar{g} \ar@{|->}[r] & \id_{W\otimes \integ/\ell}.
  }
\]
Suppose $g \in G(\integ_\ell)$, and let $\alpha$ be any eigenvalue of
$\rho_W(g)$.  On one hand, $\alpha$ is an $m^{th}$ root of unity for
some $m|[G:G^\circ]$.  On the other hand, by the commutativity of the
diagram, $\alpha \equiv 1 \bmod \ell$.  Consequently, $\alpha=1$.  Since $\rho_W(g)\in \GL_W(\integ_\ell)$ has finite order it is semisimple, and we conclude that
$\rho_W(g) = \id_W$.
\end{proof}

\begin{lemma}
  \label{L:oneellinfinity}
  Let $G/\integ[1/N]$ be a smooth affine group scheme with reductive connected component of identity.  Let $V$ be
a free $\integ[1/N]$-module of finite rank, and let $\rho\colon G \to \GL_V$
be a homomorphism of algebraic groups.  Let $\mathtt G_{\ell_0^\infty}$
be a Zariski dense subgroup of $G_{\rat_{\ell_0}}$.  If $r(\mathtt
G_{\ell_0^\infty},\rho) \ge r$, then $r_{\rat_\ell}(G,\rho) \ge r$ for $\ell$
in a set of positive density.  If $G/G^\circ$ is split, then this
holds for all $\ell$.
\end{lemma}

\begin{proof}
  By Lemma \ref{L:rZlGZl}, $r_\ell(G^\circ,\rho) \ge r$ for all
  $\ell$.  Using the same technique as in the proof of Lemma \ref{L:infmodell} to move from $G^\circ$ to $G$, we find that
  $r_\rat(G,\rho) \ge r$; for $\ell$ in a set of positive density, $r_\ell(G,\rho) \ge r$; and
  that this holds for all sufficiently large $\ell$ if $G/G^\circ$ is
  split.  In particular, by Lemma \ref{L:r_ell_implies_r_Q}(c), $\dim V_\rat^{G^\circ_\rat} \ge r$.

  Now let $\ell$ be any prime for which $r_\ell(G,\rho) \ge r$ and
  $\ell\nmid [G:G^\circ]$; by Lemma \ref{L:GcircandGmodell}, $r_{\rat_\ell}(G,\rho)
  \ge r$.

  \end{proof}

\section{Torsion points on abelian varieties}\label{Sect: tor_pts_on_ab.vars}

\subsection{Torsion points and Galois representations}
  
For the purpose of establishing notation, let $A/K$ be an abelian variety over a perfect field.  For a natural
number $N$, we let $K_{A, N}$
be the field of definition of the $N$-torsion of $A$.  We further let
$K_{A, \ell^\infty} = \bigcup_n K_{A, \ell^n}$, and let $K_A = \bigcup_N
K_{A, N}$ be the field obtained by adjoining the coordinates of all
torsion points of $A$.  Finally, we let $A_N = A[N](\bar K)$ be the geometric $N$-torsion, and $A_{\ell^\infty} = \cup_n A_{\ell^n}$.

For a fixed prime $\ell$, we have the usual representations
\begin{equation}
  \label{Eqn: F_l_adic_repn_defined}
  \xymatrix{
    \rho_{A/K,\ell}:\gal(K) \ar[r]& \GL(A_\ell) \\
    \rho_{A/K,\ell^\infty}: \gal(K) \ar[r]& \GL(T_\ell A)
    }
\end{equation}
with respective images $\Gamma_{A/K,\ell}$ and
$\Gamma_{A/K,\ell^\infty}$.

\subsubsection{Independence}\label{Sect: independence}
Serre has shown that, while the $\ell$-adic representations attached
to an abelian variety are compatible, they are also independent.

For an abelian variety $A/K$ and a prime $\ell$, briefly let
$K'_{A,\ell} \coloneqq \bigcup_{\ell\nmid N}
K(A_N)$.  Say that $A/K$ has \emph{independent torsion fields} if, for each
prime $\ell$, the Galois extensions $K_{A,\ell^\infty}$ and $K'_{A,\ell}$ are
linearly disjoint over $K$.  (Note that the compositum $K_{A,\ell^\infty}
K'_{A,\ell}$ is simply $K_A$, the field generated by the torsion
points of $A$.) 
 
\begin{lemma}\label{Lem: Serre_ind_repn}
Let $A/K$ be an abelian variety over a number field.
\begin{enumerate}[label=(\alph*)]
\item There exists a finite extension $K^{\operatorname{ind}}/K$ such
  that $A/K^{\operatorname{ind}}$ has independent torsion fields.
\item If $L/K^{\operatorname{ind}}$ is any algebraic extension, then
  $A/L$ has independent torsion fields.
\end{enumerate}
\end{lemma}
\begin{proof}
See  \cite[Th\'eor\`eme~1 and \S3]{Serre-ind-repn-2013} or \cite[\S1]{BGP-indep-l-adic-repn-geo}.
\end{proof}

\begin{lemma}\label{Lem: ind_tor_flds}
Let $A$ and $B$ be abelian varieties over a number field
$K$, and
suppose that $A\times B$ has independent torsion fields.  Then
   \[
A[\ell^\infty](K_B) = A[\ell^\infty](K_{B,\ell^{\infty}}).
\]
          \end{lemma}

\begin{proof}
It suffices to show that $A[\ell^\infty](K_B) \subset A[\ell^\infty](K_{B,\ell^{\infty}})$. Assume that $P\in A[\ell^\infty](K_{B})$, we denote by $K(P)$ the extension of $K$ by adjoining the coordinates of $P$. Then $K(P)\subset K_{A\times B, \ell^{\infty}}$. And we also have $K(P)\subset K_B=K_{B, \ell^{\infty}}\cdot K'_{B, \ell}$. Notice that $A\times B$ has independent torsion fields, so $K\subset K_{A\times B, \ell^{\infty}}\cap K'_{B, \ell}\subset K_{A\times B, \ell^{\infty}}\cap K'_{A\times B,\ell}=K$, which tells that every inclusion here is actually an equality. Hence one has that 
\[
K(P)\subset K_{A\times B,\ell^{\infty}}\cap K_B=K_{A\times B,\ell^{\infty}}\cap K_{B, \ell^{\infty}}=K_{B, \ell^{\infty}}.
\]
This means that $P\in A[\ell^\infty](K_{B, \ell^{\infty}})$. 
\end{proof}

\subsubsection{Connectedness}\label{Sect: connectedness}
We let $\mathcal G_{A/K,\ell}$ be the Zariski
closure of $\Gamma_{A/K,\ell^\infty}$ in $\GL_{H_1(A_{\overline{K}}, \rat_{\ell})} = \GL_{T_\ell A \otimes \rat_\ell}$, with connected component  $\mathcal G_{A/K,\ell}^{\circ}$. In general, $\mathcal G_{A/K,\ell}$ does not have to be connected, but when $K$ is a number field $\mathcal G_{A/K,\ell}$ will be connected after a finite extension of $K$ which is independent of $\ell$. More precisely, 

\begin{lemma}\label{Lem: K^conn}
Suppose $K/\rat$ is a finite extension.  Then
\begin{enumerate}[label=(\alph*)]
\item The finite quotient group $\mathcal G_{A/K,\ell}/\mathcal
  G_{A/K,\ell}^\circ$ is independent of $\ell$. 
\item  There exists a finite extension $K^{\operatorname{conn}}$ of
  $K$ such that, if $L$ is any finite extension of
  $K^{\operatorname{conn}}$ and $\ell$ is any prime number, the corresponding $\mathcal G_{A/L,\ell}$ is
  connected.
\end{enumerate}
\end{lemma}
\begin{proof}
See \cite{Serre-letter-Ribet} or \cite[Proposition~6.14]{Larsen-Pink-l-independence}.
\end{proof}
(In contrast, Example \ref{Exa: torsion_infinite-11} will show  that if $K$ is algebraic but infinite, then such a finite connectedness extension need not exist.)

\subsection{Mumford--Tate conjecture}\label{Sect: MT_conj}

This section is devoted to recalling the Mumford--Tate conjecture. In particular, we will review a result of Cadoret and Moonen \cite[\S~1]{Cadoret-Moonen-int-adelic-MTC} and of Hindry and Ratazzi \cite{Hindry-Ratazzi-torsion-type-I-II} which states that as $\ell$ varies, the $\ell$-adic image of the Galois group is a bounded index subgroup of the $\Z_{\ell}$-points of the Mumford--Tate group. 

Let $K$ be a number field, embedded in $\cx$.  To ease notation slightly, we write $\mt(A)$ for the Mumford--Tate group of an abelian variety $A$ over $K$, i.e., $\mt(A)\coloneqq \mt(H_1(A_\cx, \rat))$ (cf. Section~\ref{SS:MT pair}). This is a connected $\rat$-algebraic group.  Let $G_A$ be the Zariski closure of $\mt(A)$ in $\GL_{H_1(A_\cx,\integ)}$; it is a group scheme over $\integ$.  Then $G_A$ is smooth over $\integ[1/N_A]$ for some positive integer $N_A$, and $G_{A,\rat} = \mt(A)$.

If $K =K^{\rm{conn}}$, then it is known that there is a natural inclusion
$\Gamma_{A/K,\ell^\infty} \subset \mt(A)(\rat_\ell)$, and thus an inclusion
$G_{A/K,\ell^\infty} \hookrightarrow \mt(A)_{\rat_\ell}$ of algebraic groups
over $\rat_\ell$.  The Mumford--Tate conjecture asserts that this
inclusion is actually an isomorphism. More precisely, for every prime $\ell\nmid N$, both $G_{A, \ell}\coloneqq G_A\times_{\Z[\frac{1}{N}]} {\rm Spec}\integ_\ell$ and $\mathcal {G}_{A/K, \ell}^{\circ}$ are subgroup schemes of $\GL_{H^1(A, \integ_\ell)}$. The following conjecture claims the comparison result of the two group schemes. 
\begin{conj}\cite[Mumford--Tate Conjecture]{Cadoret-Moonen-int-adelic-MTC}\label{Conj: MT_conj}
With the above notations, $G_{A, \ell}=\mathcal {G}_{A/K, \ell}^{\circ}$. 
\end{conj}

\begin{remark}\label{Rmk: Mumford--Tate}
Conjecture~\ref{Conj: MT_conj} is equivalent to the usual statement \[
\mathcal {G}_{A/K, \rat_\ell}^{\circ}=\mt(A)\times_{\rat} \rat_{\ell}
\]
for every prime $\ell$ \cite{Cadoret-Moonen-int-adelic-MTC}.
\end{remark}

In this paper, the Mumford--Tate conjecture is a standing assumption
we require in order to make any significant progress. The conjecture is known to be true for large classes of abelian varieties. For example, it is known that an absolutely simple abelian variety $A$ of dimension $g$ satisfies the Mumford--Tate conjecture in any of the following settings:
\begin{enumerate}
    \item $g$ is prime \cite{Ribet-Hodge-cycle,Tankeev-cycles-simple-ab-prime-1,Tankeev-cycles-simple-ab-prime-2};
    \item  $g \le 3$ \cite{Moonen-Zarhin-Hg-cls-av-low-dim};
    \item $\End_{\bar K}(A) = \integ$ and $g$ satisfies certain numerical conditions (for instance, $g$ is odd)  \cite{Pink-mt-cocharacter};
    \item $A$ has complex multiplication \cite{Pohlmann, Yu-MTCM}.
    \end{enumerate}
 Our list is far from complete.  See also  \cite{Vasiu-MT} and the discussion in \cite[Section~2.4.]{Moonen-families} for additional references and known results.   
Moreover, if the Mumford--Tate conjecture is true for abelian varieties $A$ and $B$, then it is also true for their product $A\times B$ \cite{commelinproductmt}.

In the
presence of the Mumford--Tate conjecture we have good control over
$\Gamma_{A/K,\ell^\infty}$.

\begin{thm}\cite[Theorem~A]{Cadoret-Moonen-int-adelic-MTC}
  \cite[Th\'eor\`eme 10.1]{Hindry-Ratazzi-torsion-type-I-II}\label{Thm: Cadoret_Moonen}
  Let $A$ be an abelian variety defined over $K$, assume $K=K^{\rm
    conn}$, and assume that the Mumford--Tate conjecture is true for
  $A$. Then the index $[G_A(\integ_\ell): \Gamma_{A/K, \ell^\infty}]$ is
  bounded when $\ell$ varies. 

In particular, $\set{\Gamma_{A/K,\ell}}$ is a collection of bounded
subgroups of $G_A$.
\end{thm}

\section{Torsion-finite pairs of abelian varieties}

\subsection{Mumford--Tate groups for a pair of abelian
  varieties}\label{SS:MT pair}

Let $A$ and $B$ be abelian varieties over a subfield $K$ of $\cx$.  Let
 $G_A$, $G_B$ and $G_{A\times B}$ denote the ($\integ$-models of) the Mumford--Tate groups
of, respectively, $A$, $B$ and $A\times B$, and let $sG_A$, $sG_B$ and
$sG_{A\times B}$ denote their respective Hodge groups.  Recall that the
Mumford--Tate group $G_C$ of a complex abelian variety $C$ is the $\rat$-algebraic hull of the morphism $\Res_{\cx/\real}\gp_m \to \aut(H_1(C,\rat)\otimes_\rat\real)$ defining the Hodge structure on $H_1(C,\rat)$.  Equivalently, it is the Tannakian fundamental group (really, the group which represents automorphisms of the fiber functor which sends a Hodge structure $V$ to its underlying vector space $|V|$) of $\ang{H_1(C,\rat)}$, the tensor
category generated by the Hodge structure $H_1(C,\rat)$. Since $H^1(C,\rat)$ is dual to $H_1(C,\rat)$, the tensor categories $\ang{H_1(C,\rat)}$ and $\ang{H^1(C,\rat)}$ coincide, and we may use either description to compute $\mt(C)$.  We have $G_C/sG_C \iso \gp_m$.

Since
$H_1(A\times B,\rat) \iso H_1(A,\rat)\oplus H_1(B,\rat)$ is an object
of $\ang{H_1(A,\rat),H_1(B,\rat)}$, there is a canonical inclusion 
$\iota\colon G_{A\times B} 
 \hookrightarrow G_A\times G_B$.  Moreover, $H_1(A,\rat)$ and $H_1(B,\rat)$ are both objects of $\ang{H_1(A\times B,\rat)}$.  The corresponding inclusions $\ang{H_1(A,\rat)}
\hookrightarrow \ang{H_1(A\times B,\rat)}$ and
$\ang{H_1(B,\rat)}\hookrightarrow \ang{H_1(A\times B,\rat)}$ yield
surjections $G_{A\times B} \twoheadrightarrow G_A$ and $G_{A\times
  B}\twoheadrightarrow G_B$.  Thus, the three algebraic groups $G_A$,
$G_B$ and $G_{A\times B}$ satisfy the hypotheses of Goursat's lemma (Lemma \ref{L:goursat}), and fit in a diagram as follows.
\begin{equation}
\label{D:MT goursat}
\xymatrix{
G_{A\times B} \ar@{^(->}[r]^\iota & G_A \times G_B \ar@{->>}[dl]^{\pi_A} \ar@{->>}[dr]_{\pi_B} \\
G_A && G_B
}
\end{equation}
Let $M_{A,B} = \ker(\pi_B\circ\iota)$;  under the isomorphism
$G_A\times \set e \iso G_A$, it is isomorphic to a normal algebraic
subgroup $H_{A,B}$ of $G_A$.  Define $M_{B,A}$ and $H_{B, A}$ in an
analogous fashion.  Because $H_1(A,\rat)$ and $H_1(B,\rat)$ have the
same nonzero weight, $H_{A,B} \subset sG_A$ and $H_{B,A} \subset
sG_B$.  Consequently, the Hodge groups also satisfy the hypotheses of
Goursat's lemma, i.e., fit in a diagram
\begin{equation*}
\xymatrix{
sG_{A\times B} \ar@{^(->}[r]^\iota & sG_A \times sG_B \ar@{->>}[dl]^{\pi_A} \ar@{->>}[dr]_{\pi_B} \\
sG_A && sG_B
}
\end{equation*}
In particular, $sG_{A\times B}$ is the inverse image in $sG_A\times sG_B$ of the graph of an isomorphism
\begin{equation}
  \label{E:special hodge goursat}
\xymatrix{  \frac{sG_A}{H_{A,B}} \ar[r]^\sim & 
  \frac{sG_B}{H_{B,A}}}.
\end{equation}

For future use, we record the following observation.

\begin{lemma}
\label{L:MT isomorphic}  Let $A$ and $B$ be complex abelian varieties.  The following are equivalent.
  \begin{enumerate}[label=(\alph*)]
  \item $A$ and $B$ are isogenous;
  \item $H_1(A,\rat)$ and $H_1(B,\rat)$ are isomorphic representations of $\smt(A \times B)$;
    \item The canonical surjections $\smt(A\times B) \twoheadrightarrow \smt(A)$ and $\smt(A\times B) \twoheadrightarrow \smt(B)$ are isomorphisms, and $H_1(A,\rat)$ and $H_1(B,\rat)$ are isomorphic representations of this common group.
    \end{enumerate}
  \end{lemma}
  
\begin{proof}
      The category $\ang{H_1(A\times B,\rat)}$ is equivalent to the
      category of representations of $\mt(A\times B)$.   So
      $H_1(A,\rat)$ and $H_1(B,\rat)$ are isomorphic in the category
      of Hodge structures, or equivalently in the full subcategory
      generated by $H_1(A\times B,\rat)$, if and only if they are
      isomorphic representations of $\mt(A\times B)$.  For weight
      reasons, it suffices to verify this for the Hodge group
      $\smt(A\times B)$.  Riemann's theorem -- that the isogeny class
      of an abelian variety is determined by its Hodge structure --  proves the equivalence of (a) and (b).

 If $A$ and $B$ are isogenous, it's well-known that $\mt(A\times B) \iso \mt(A) \iso \mt(B)$ (e.g., \cite[Rem.~1.8]{Moonen-MT}).  Conversely, under the hypothesis of (c), weight considerations show that the corresponding hypothesis holds for Mumford--Tate gorup, too.  Now use the fact that $\mt(A)$ is canonically isomorphic to the image of $\mt(A\times B)$ is $\GL_{H_1(A,\rat)}$ and the analogous statment for $B$ in order to deduce (b).
 \end{proof}

Now suppose that $A$ and $B$ have complex multiplication (see \S \ref{Sect: CM} for a review of this concept). Then $A\times B$ does, too, and the Mumford--Tate groups $G_A$, $G_B$ and $G_{A\times B}$ are all tori.  Taking character groups in \eqref{D:MT goursat} yields a diagram of $\integ$-modules
\begin{equation}
\label{D:MT CM goursat}
\xymatrix{
X^*(G_{A\times B})  & X^*(G_A) \times X^*(G_B) \ar@{->>}[l]  \\
X^*(G_A) \ar@{^(->}[ur] && X^*(G_B) \ar@{_(->}[ul]
}.
\end{equation}
In particular, we may use this diagram to compute $H_{A, B}$, a group of multiplicative type; it is the group whose character group is 
\begin{align}
\label{E:HAB for tori}
X^*(H_{A,B}) &= \frac{X^*(G_{A\times B})}{X^*(G_B)}.
\intertext{If we identify $X^*(G_A)$ and $X^*(G_B)$ with their images under, respectively, the inclusions $(\pi_A\circ\iota)^*$ and $(\pi_B\circ\iota)^*$, we may rewrite this as}
\label{E:HAB for tori reprise}
 X^*(H_{A,B}) &= \frac{X^*(G_A)+X^*(G_B)}{X^*(G_B)} \iso \frac{X^*(G_A)}{X^*(G_A)\cap X^*(G_B)}.
  \end{align}

\subsection{Galois representations for a pair of abelian varieties}

Now further suppose that $K$ is finitely generated over $\rat$, and assume that $A$ and $B$ satisfy the Mumford--Tate conjecture.

Since our main results concern potentially
infinite torsion, we will assume that $A\times B$ has connected,
independent Galois representations.

For a positive integer $N$, we identify the Galois group $\Gamma_{A/K,N}$ with a subgroup
$\mathtt G_{A,N}$ of $G_A(\integ/N)$, and make similar
identifications of the image of $\gal(K)$ acting on the
$N$-torsion of $B$ and of $A\times B$.  The $N$-torsion fields
of $A$ and $B$ are then arranged in the following tower, where each
extension is labeled with its corresponding Galois group.
  
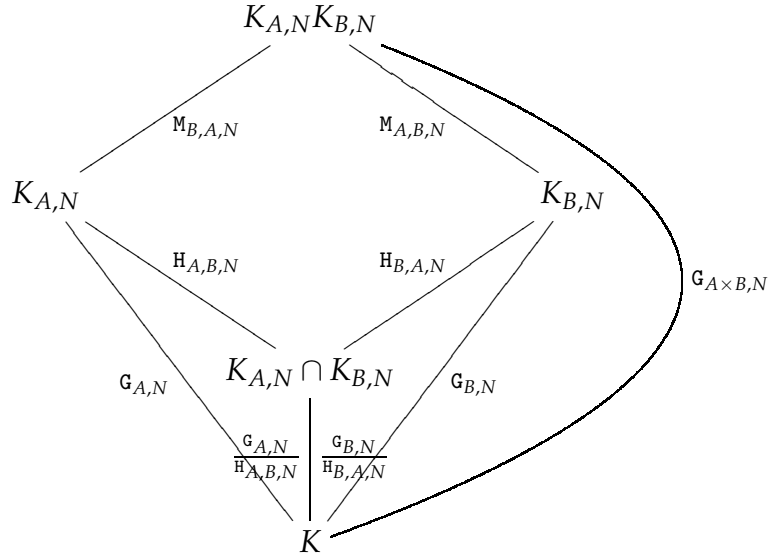
\[
    \xymatrix@+=4pc{
      &K_{A,N}K_{B,N} \ar@{-}[dl]^{\mathtt M_{B,A,N}}
      \ar@{-}[dr]_{\mathtt M_{A,B,N}} \ar@{-}@/^25ex/[ddd]^{\mathtt
        G_{A\times B,N}}&\\
K_{A,N} \ar@{-}[dr]^{\mathtt H_{A,B,N}} \ar@{-}[ddr]_{\mathtt
  G_{A,N}} && K_{B,N} \ar@{-}[dl]_{\mathtt H_{B,A,N}}
\ar@{-}[ddl]^{\mathtt G_{B,N}} \\
& K_{A,N}\cap K_{B,N} \ar@{-}[d]_{\frac{\mathtt
    G_{A,N}}{\mathtt H_{A,B,N}}}^{\frac{\mathtt
    G_{B,N}}{\mathtt H_{B,A,N}}} \\
&K
}
\]
Let $\mathtt H_{A,B,\ell^\infty} = \varprojlim_{n}\mathtt H_{A,B,\ell^n}
\subset H_{A,B}(\integ_\ell)$.

We have $\gal(K_{B,\ell} K_{A,\ell}/K_{B,\ell}) =
\mathtt M_{A,B,\ell} \iso \mathtt H_{A,B,\ell} \subseteq
H_{A,B}(\ff_\ell)$, and $A[\ell](K_{B,\ell})$ is the set of elements
of $A_\ell$ fixed by $\mathtt H_{A,B,\ell}$.

\begin{lemma}
\label{L:ell is enough}
    Let $A$ and $B$ be abelian varieties over a number field $K$.  Suppose that $A\times B$ has independent Galois representations.  Then for each prime $\ell$,
$A[\ell](K_B)$ is nontrivial if and only if $A[\ell](K_{B,\ell})$ is nontrivial.
\end{lemma}

\begin{proof}
By independence, $A[\ell](K_B) = A[\ell](K_{B,\ell^\infty})$ (Lemma \ref{Lem: ind_tor_flds}).  Suppose that, for some $n > 1$, $\#A[\ell](K_{B,\ell^n}) > 1$; equivalently, $\gal(K_{B,\ell^n})$ has a nontrivial fixed point in $A_\ell$.  By Lemma \ref{L:fixed by index ell subgroup}, in order to show that $\gal(K_{B,\ell})$ has a nontrivial fixed point in $A_\ell$, it suffices to show that $[\gal(K_{B,\ell^n}):\gal(K_{B,\ell})]$ is a power of $\ell$.  This last claim follows from the inclusions
\[\xymatrix{
\frac{\gal(K_{B,\ell})}{\gal(K_{B,\ell^n})} \iso \gal(K_{B,\ell^n}/K_{B,\ell}) \ar@{^(->}[r] & \set{ g \in \aut(B_{\ell^n}): g \equiv \operatorname{id} \bmod \ell}\subset 1 + \ell \operatorname{End}(B_{\ell^n})}.
\]
\end{proof}

\begin{lemma}
\label{L:bounded image for KB}
Let $A$ and $B$ be abelian varieties over a number field $K$.  Suppose
that $A$ and $B$ satisfy the Mumford--Tate conjecture, and that
$A\times B$ has connected Galois representations.  Then
\begin{enumerate}[label=(\alph*)]
  \item $\set{
\mathtt H_{A,B,\ell}\cap H_{A,B}^\circ(\ff_\ell)}$ is a collection of
bounded subgroups of $H_{A,B}^\circ$; and
\item  $\mathtt H_{A,B,\ell^\infty}\cap H_{A,B}^\circ(\integ_\ell)$ is
    Zariski dense in $H_{A,B,\rat_\ell}^\circ$.
  \end{enumerate}
\end{lemma}

\begin{proof}
Since $A\times B$ satisfies the Mumford--Tate conjecture,
$\set{\mathtt G_{A\times B,\ell}}$ is bounded in $G_{A\times B}$ (Theorem \ref{Thm: Cadoret_Moonen}).  Now apply Lemma 
\ref{L:abstractboundedindex}(b) to the exact sequence
\[
  \xymatrix{
    0 \ar[r] & M_{A,B} \ar[r] & G_{A\times B} \ar[r] & G_B \ar[r] & 0
  }\]
to deduce (a).

For a fixed $\ell$, there exists an $n = n_\ell$ such that $\mathtt G_{A,\ell^\infty}$ contains $\ker(G_A(\integ_\ell)\to G_A(\integ_\ell/\ell^n))$, and so $\mathtt H_{A,B,\ell^\infty}\cap H_{A,B}^\circ(\integ_\ell)$ contains $\ker(H_{A,B}^\circ(\integ_\ell) \to H_{A,B}^\circ(\integ_\ell/\ell^n))$.  Then counting points (with values in $\integ_\ell/\ell^{m+n}$ for $m\gg 0$) shows that the Zariski closure of $\mathtt H_{A,B,\ell^\infty}\cap H_{A,B}^\circ(\integ_\ell)$, a closed subgroup of the irreducible variety $H_{A,B,\rat_\ell}^\circ$, must be all of $H^\circ_{A,B,\rat_\ell}$.
\end{proof}

\subsection{Preliminaries on torsion finiteness}

If two abelian varieties are isomorphic, or more generally isogenous, then it is easy to see that each is torsion infinite for the other:

\begin{lemma}\label{L:isogenous}
  Let $A$ and $B$ be abelian varieties over a number field $K$.
  \begin{enumerate}[label=(\alph*)]
    \item If $A$ and $A'$ are isogenous over $K$, and if $B$ and $B'$ are
      isogenous over $K$, then $A$ is torsion finite for $B$ over $K$ if and only if
      $A'$ is torsion finite for $B'$ over $K$.
    \item If $L/K$ is a finite extension, and if $A_L$ is torsion
        finite for $B_L$ over $L$, then $A$ is torsion finite for $B$ over $K$.
    \item If $m$ and $n$ are two positive integers, then $A$ is torsion finite for $B$ over $K$ if and only if $A^{m}$ is torsion finite for $B^{n}$ over $K$.
      \end{enumerate}
    \end{lemma}

    \begin{proof}

Let $g\colon B \to B'$ be an isogeny of exponent $N$; there is an isogeny $g'\colon B' \to B$ such that $g'\circ g = [N]_B$. Then for any (not necessarily finite) field extension $F/K$ one has
    \[
    \frac 1 N \cdot \#B'(F)_{\rm tors}\leq \#B(F)_{\rm tors}\leq N \cdot \#B'(F)_{\rm tors}.
    \]
    In particular, $B(F)_{\rm tors}$ and $B'(F)_{\rm tors}$ are either both finite or infinite. Moreover, one can also deduce that $K_B=K_{B'}$. We deduce (a) after applying the same argument to $A$ and $A'$.
    
    Part (b) is obvious since $L_B$ contains $K_B$. 
    
    Part (c) follows from the observation that $K_{C,N}=K_{C^{r}, N}$ for any abelian variety $C/K$ and any natural numbers $r$ and $N$.
  \end{proof}

  \begin{lemma}
    \label{L:tf for products}
    Suppose $A$ and $B$ are abelian varieties over a field $K$.  There exist a finite extension $L$ and an isogeny of $L$-abelian varieties $\oplus_{i=1}^r A_i^{m_i} \to A_L$ with each $A_i$ absolutely simple; and $A$ is essentially torsion finite for $B$ if and only if each $A_i$ is essentially torsion finite for $B_L$.
  \end{lemma}

  \begin{proof}
    The existence of such an $L$ and a factorization of $A_L$ is standard; since we are only concerned with essential torsion finiteness, we may and do assume $L=K$.  Then $A$ is essentially torsion finite for $B$ if and only if $\oplus_i A_i^{m_i}$ is (Lemma \ref{L:isogenous}(a)), which obviously holds if and only if each summand $A_i^{m_i}$ is essentially torsion finite for $B$.  By Lemma \ref{L:isogenous}(c), this holds if and only if each $A_i$ is essentially torsion finite for $B$.
  \end{proof}

\subsection{Potentially torsion infinite pairs}
\label{Sect: potentially-infinite}

If $H_{A,B}$ is connected and if $A$ acquires infinite torsion over
$K_B$, then $A$ acquires $\ell$-power torsion for all $\ell$:
\begin{lemma}
  Let $A$ and $B$ be abelian varieties over a number field $K$.
  Suppose that $A$ and $B$ satisfy the Mumford--Tate conjecture, and that
  $A\times B$ has connected independent Galois representations.

  Suppose that $H_{A,B}$ is
  connected.  Then the following are equivalent.

  \begin{enumerate}[label=(\alph*)]
  \item $A(K_B)_{\tors}$ is infinite.
  \item For all $\ell$, $A[\ell](K_B)$ is nontrivial.
  \item For all $\ell$, $A[\ell^\infty](K_B)$ is infinite.
  \end{enumerate}
\end{lemma}

\begin{proof}
It suffices to show that (a) implies each of (b) and (c). Note that,
since $A\times B$ has independent representations,
$A[\ell^\infty](K_B) = A[\ell^\infty](K_{B,\ell^\infty})$.

By Lemma \ref{L:bounded image for KB}, $\set{\mathtt H_{A,B,\ell}}$
is a collection of bounded subgroups of the connected group
$H_{A,B}$, and each $\mathtt H_{A,B,\ell^\infty}$ is Zariski dense in
$H_{A,B,\rat_\ell}$.

 Suppose that (a) holds; then 
$A[\ell](K_{B,\ell^\infty})$ is nontrivial for infinitely many
$\ell$, or
$A[\ell_0^\infty](K_{B,\ell_0^\infty})$ is infinite for some
$\ell_0$.  In the former case, $A[\ell](K_{B,\ell})$ is nontrivial for infinitely many $\ell$ (Lemma \ref{L:ell is enough}), and thus
$r_\ell(H_{A,B},\rho_A)$ is positive for infinitely many
$\ell$; in the latter,
$r_{\rat_{\ell_0}}(H_{A,B},\rho_A)$ is positive.  Thus by Lemma
\ref{L:rlGl} or \ref{L:rZlGZl},
$\rho_\rat(H_{A,B},\rho_A)$ is positive; therefore, so is
$\rho_\ell(H_{A,B},\rho_A)$ and
$\rho_{\rat_{\ell}}(H_{A,B},\rho_A)$ for each
$\ell$.  Therefore, both (b) and (c) hold.
\end{proof}

In the absence of a connectedness hypothesis on $H_{A,B}$, our results
are less balanced.  Moreover, the Mumford--Tate conjecture doesn't immediately imply that $\mathtt H_{A,B,\ell}$ meets every geometrically irreducible component of $H_{A,B,\ell}$.  In situations where this is known, however, we can deduce the following statement.

\begin{lemma}
\label{L:TI for one good ell}
  Let $A$ and $B$ be abelian varieties over a number field $K$.
  Suppose that $A$ and $B$ satisfy the Mumford--Tate conjecture and
  that $A\times B$ has connected, independent Galois representations.

  Suppose that $A(K_B)_{\tors}$ is infinite, and that there exists some $\ell_0$ such that $A[\ell_0](K_{B,\ell_0})$ is nontrivial. Additionally, assume that $\mathtt H_{A,B,\ell_0}$ meets every geometrically irreducible component of $H_{A,B,\ell_0}$, and $H_{A,B,\integ_{\ell_{0}}}$ is smooth.
    Then for $\ell$ in a set of positive density, $A[\ell](K_B)$ is
  nontrivial and $A[\ell^\infty](K_B)$ is
  infinite.
\end{lemma}

\begin{proof}
If $A[\ell](K_B)$ is nontrivial for infinitely many $\ell$, this follows from Lemma \ref{L:infmodell}, applied to the collection of bounded subgroups $\set{\mathtt H_{A,B,\ell}}$ of $H_{A,B}$.  If instead there exists some $\ell_1$ such that $A[\ell_1^\infty](K_B)$ is infinite, then as in the proof of Lemma \ref{L:oneellinfinity} we find that $r_\rat(H_{A,B}^\circ,\rho_A)$ is positive; and that for $\ell$ in a set of positive density (namely, the set of $\ell$ relatively prime to $[H_{A,B}:H_{A,B}^\circ]$ and with the same Artin symbol as $\ell_0$ in some finite splitting field for $H_{A,B}/H_{A,B}^\circ$), $r_\ell(H_{A,B},\rho_A)$ is also positive.
\end{proof}

The hypothesis on $\ell_0$ in Lemma \ref{L:TI for one good ell} seems difficult to work with abstractly, although in explicit examples one can compute $H_{A,B}/H_{A,B}^\circ$ (e.g., Example \ref{Exa: torsion_infinite-11}) and thereby make progress.  However, we can still make a uniform statement purely in terms of torsion,  at the cost of surrendering some control of the precise field over which $A$ acquires infinite $\ell$-torsion.

\begin{lemma}
\label{L:newmainHabarbitrary}
  Let $A$ and $B$ be abelian varieties over a number field $K$.
  Suppose that $A$ and $B$ satisfy the Mumford--Tate conjecture and
  that $A\times B$ has connected independent Galois representations.

  Suppose that $A(K_B)_{\tors}$ is infinite.  Let $N_{A,B} =
  [H_{A,B}:H^\circ_{A,B}]$.

 \begin{enumerate}[label=(\alph*)]
  \item We have $r_\rat(H^\circ_{A,B},\rho_A)>0$. 
    \item For each $\ell$ there exists a finite extension $\widetilde{K_{B,\ell^\infty}}$ of
      $K_{B,\ell^\infty}$ such that $[\widetilde{K_{B,\ell^\infty}}:
      K_{B,\ell^\infty}] | N_{A,B}$ and
      $A[\ell^\infty](\widetilde{K_{B,\ell^\infty}})$ is infinite.
      \item For each $\ell$ there exists some $n_\ell$ such that $A[\ell^\infty](K_{B,\ell^\infty}K_{A,\ell^{n_\ell}})$ is infinite; and if $\ell \nmid N_{A,B}$, then we may take $n_\ell=1$.  (In fact, $K_{B,\ell^\infty}\subseteq \widetilde{K_{B,\ell^\infty}}\subseteq K_{B,\ell^\infty}K_{A,\ell^{n_\ell}}$.)
   \end{enumerate}
\end{lemma}

\begin{proof}
Since $A(K_B)_{\tors}$ is infinite, there exist infinitely many
$\ell$ such that $A[\ell](K_B)$ is nontrivial, or there is some
$\ell_0$ such that $A[\ell_0^\infty](K_B)$ is infinite .  By Lemma
\ref{L:infmodell} or \ref{L:rZlGZl} as appropriate,
$r := r_\rat(H_{A,B}^\circ,\rho_A)>0$.

Now fix a prime $\ell$.  Let $\widetilde{K_{B,\ell^\infty}}$ be the smallest
extension of $K_{B,\ell^\infty}$ for which
$\rho_{A,\ell^\infty}(\gal(\widetilde{K_{B,\ell^\infty}})) \subseteq
H^\circ_{A,B}(\integ_\ell)$.  Then
$[\widetilde{K_{B,\ell^\infty}}:K_{B,\ell^\infty}] | N_{A,B}$, and
$\rank_{\integ_\ell} A[\ell^\infty](\widetilde{K_{B,\ell^\infty}}) \ge r>0$.

Moreover, we have inclusions $K_{B,\ell^\infty} \subseteq
\widetilde{K_{B,\ell^\infty}} \subseteq (K_{B,\ell^\infty})_{A,\ell^\infty}$.
Since the first extension is finite, there exists some $n$ such that
$\widetilde{K_{B,\ell^\infty}} \subseteq (K_{B,\ell^\infty})_{A,\ell^n}$.  Now, if
$n\ge 2$, then 
$\gal((K_{B,\ell^\infty})_{A,\ell^n} / (K_{B,\ell^\infty})_{A,\ell})$
is a group whose order is a power of $\ell$.  Consequently, if $\ell\nmid N_{A,B}$,
then $\widetilde{K_{B,\ell^\infty}} \subseteq (K_{B,\ell^\infty})_{A,\ell}=
K_{A,\ell}K_{B,\ell^\infty}$.
\end{proof}

\begin{remark}
In the context of Lemma \ref{L:newmainHabarbitrary}(b), one might hope that there is a finite extension
$L_B$ of $K_B$ such that $A[\ell](L_B)$ is nontrivial for each
$\ell$.  (Even more optimistically, one might hope that such an $L_B$
is the compositum of $K_B$ and a finite extension of $K$.)  However, there is no independence-of-$\ell$ connectedness result for Galois representations of infinite algebraic extensions of $\rat$ to which one might appeal.  In fact, we will see below (\ref{Exa: torsion_infinite-11}) that in general, no such uniform-in-$\ell$ finite extension exists.  In this sense, Lemma \ref{L:newmainHabarbitrary}(b) is optimal without additional hypotheses.
\end{remark}

\begin{thm}
  \label{T:A abs simp}
  Let $A$ and $B$ be abelian varieties over a number field $K$.
  Suppose that $A$ and $B$ satisfy the Mumford--Tate conjecture, and
  that $A$ is absolutely simple.  Then the following are equivalent:
  \begin{enumerate}[label=(\alph*)]
  \item $A$ is potentially torsion infinite for $B$;
  \item $\dim H_{A,B} = 0$;
  \item $\dim G_{A\times B} = \dim G_B$;
  \item $\rank G_{A\times B} = \rank G_B$; and
    \item there exists a finite extension $K'$ of $K$ such that for
      all sufficiently large $\ell$, $K'_{A,\ell^\infty} \subset K_B
      K'_{A,\ell}$, and thus $A[\ell^\infty](K_B K'_{A,\ell}) = A_{\ell^\infty}$.
  \end{enumerate}
\end{thm}

\begin{proof}
  Suppose that $A$ is potentially torsion infinite for $B$.  Then,
  possibly after replacing $K$ with a finite extension, we find that
  $r := r_\rat(H^\circ_{A,B},\rho_A)>0$ (Lemma
  \ref{L:newmainHabarbitrary}(a)).  Note that $H_{A,B}$, and thus
  $H^\circ_{A,B}$, are normal subgroups of $\mt(A)$ (Lemma \ref{Lem:
    component is normal}).

Because $A$ is absolutely simple and $\mt(A)$ is reductive,
$H_1(A,\rat)$ is an irreducible representation of $\mt(A)$. Now, $H_1(A,\rat)^{H^\circ_{A,B}}$ is a
sub-$\mt(A)$-representation (Lemma \ref{Lem: fixed_subsp_sub_repn}) of
$H_1(A,\rat)$.  Since $\dim H_1(A,\rat)^{H_{A,B}^\circ} = r >0$ and $H_1(A,\rat)$ is irreducible, it follows that $H_1(A,\rat)^{H_{A,B}^\circ} =
H_1(A,\rat)$.  This implies that $H_{A,B}^\circ$ is trivial, and thus
$\dim H_{A,B} = 0$.

The converse, that (b) implies (a),  is easy.  Indeed, suppose $\dim H_{A,B} = 0$, and fix a prime
$\ell$.  Then $\gal(K_B)$ acts on $T_\ell A$ through a subgroup of
$H_{A,B}(\integ_\ell)$, which is by hypothesis a finite group; after
passage to a finite extension, the Galois group acts trivially on all
$\ell$-power torsion points.

The equivalence of (b), (c), and (d) is a standard observation about
reductive groups (Lemma \ref{Lem: dim_and_rank}).

Now suppose (b) holds.   After replacing $K$ with a suitable finite
extension, we may and do assume that $A\times B$ has connected
independent Galois representations. Then Lemma
\ref{L:newmainHabarbitrary}(c) shows that there exists some $n_\ell$
such that $\gal(K_{B} K_{A,\ell^{n_\ell}})$ acts trivially on $T_\ell
A$; and that if $\ell \nmid [H_{A,B}:H^\circ_{A,B}]$, then we may
take $n_\ell = 1$.

The proof is completed with the trivial observation that (e) implies (a).
\end{proof}

\begin{remark}\label{R:A isotypic}
  In Theorem \ref{T:A abs simp}, by Lemma \ref{L:tf for products}, it
  suffices to assume that $A$ is absolutely isotypic.
\end{remark}

\begin{cor}
\label{C: TF cm elliptic curve}
Let $B/K$ be an abelian variety over a number field for which the
Mumford--Tate conjecture holds, and let $E/K$ be an elliptic curve
with complex multiplication.  Then either $E$ is potentially torsion
infinite for $B$ or $sG_{E\times B} = sG_E \times sG_B$.
\end{cor}

\begin{proof}
The Mumford--Tate conjecture holds for $E$ and thus for $E\times B$, and $E$ is visibly absolutely simple;  it therefore suffices to show that if $E$ is essentially torsion finite for
  $B$ then the special Mumford--Tate group of the product is the
  product of the special Mumford--Tate groups.  By Theorem \ref{T:A abs simp},
$\dim H_{E,B}>0$.  Since $H_{E,B}$ is a positive-dimensional subgroup
of the one-dimensional torus $sG_E$, it follows that $H_{E,B} =
sG_E$, and thus $sG_{E\times B} = sG_E \times sG_B$ (Remark
\ref{R:goursat product}).
\end{proof}

\begin{cor} Let $A$ and $B$ be abelian varieties over a number field, of respective dimensions $d_A$ and $d_B$.  Suppose that $A$ and $B$ satisfy the Mumford--Tate conjecture, that $A$ is absolutely simple, and that
\begin{equation}
\label{E:dimension relation}
\log_2 d_A \ge 3 d_B-1.
\end{equation}
Then $A$ is essentially torsion finite for $B$.
\end{cor}

\begin{proof}
Let $r_A$ and $r_B$ denote the respective ranks of the Mumford--Tate groups of $A$ and $B$.  On one hand, we have the trivial bound $r_B \le d+1$.  On the other hand, a weak form of \cite[Thm. 1.2]{orr15} implies that $r_A \ge \frac 13(\log_2 d_A+2)$.  Therefore, hypothesis \eqref{E:dimension relation} implies that $r_A > r_B$.  Since $\rank(G_A)-\rank(G_B) = \rank(H_{A,B})-\rank(H_{B,A})$, we find that the rank of $H_{A,B}$ is positive, and thus $\dim H_{A,B}>0$.  Now apply Theorem \ref{T:A abs simp}.
\end{proof}

\begin{cor}
\label{C:mutually torsion infinite}
Suppose that $A$ and $B$ are two absolutely simple abelian varieties
over $K$, and that the Mumford--Tate conjecture holds for $A\times
B$.  Then the following are equivalent.
\begin{enumerate}[label=(\alph*)]
\item $A$ and $B$ are mutually potentially torsion infinite.
  \item The natural surjections $G_{A\times B} \to G_A$ and
    $G_{A\times B} \to G_B$ are isogenies.
  \item The natural surjections $sG_{A\times B} \to sG_A$ and
    $sG_{A\times B} \to sG_B$ are isogenies.
  \item $\dim(G_{A\times B}) = \dim G_A = \dim G_B$;
  \item $\dim(sG_{A\times B}) = \dim sG_A = \dim sG_B$.
  \item $\rank(G_{A\times B}) = \rank G_A = \rank G_B$;
  \item $\rank(sG_{A\times B}) = \rank sG_A = \rank sG_B$.
  \end{enumerate}
\end{cor}

At the opposite extreme, we have:

\begin{cor}
\label{C:hodge group is product}
Suppose that $A$ and $B$ are two absolutely simple abelian varieties
over $K$, and that the Mumford--Tate conjecture holds for $A\times
B$.  If $sG_{A\times B} = sG_A \times sG_B$, then $A$ and $B$ are
mutually essentially torsion finite.
\end{cor}

As we noted in the introduction, Ichikawa \cite{Ichikawa-Hg-gp-AV} and Lombardo \cite{Lombardo-Hg-gp-pair-AV-2016} have given sufficient criteria for a pair of abelian varieties $(A,B)$ to satisfy $sG_{A\times B} = sG_A \times sG_B$; a typical example is when $A$ and $B$ are non-isogenous abelian varieties of odd dimension with absolute endomorphism ring $\integ$.

\section{Applications}\label{Sect: application}
  
In this section we apply Theorem \ref{T:A abs simp} to the following two classes of abelian varieties:
\begin{enumerate}
    \item[(a)] CM abelian varieties;
    \item[(b)] Abelian varieties with dimension smaller than 4.
\end{enumerate}
Since the Mumford--Tate groups, and in fact the Mumford--Tate conjecture, are known for these classes of varieties, we are able to obtain some unconditional results.

More precisely, in Section \ref{Sect: CM}, after a brief review of CM abelian varieties, we give a criterion (Theorem~\ref{Thm:
  main_thm_CM_pair}) to decide the essential torsion finiteness of a pair
of CM abelian varieties in terms of their CM-types. This criterion,
compared with Theorem~\ref{T:A abs simp}, has
the advantage of being effective. In particular,  one can use it to create
examples of non-isogenous but potentially torsion infinite
pairs. This is demonstrated in  Examples~\ref{Exa: torsion_infinite-7} and
\ref{Exa: torsion_infinite-11}. Moreover, since the Mumford--Tate group of a CM abelian variety is a torus -- indeed, this characterizes CM abelian varieties -- our theorem
indicates a way to describe the extra Hodge classes on the product of CM abelian varieties via relations among the characters of their Mumford--Tate groups.   We also explore the
relation between extra Hodge classes and torsion infiniteness in the CM
case; this is the main content of Section
\ref{Sect: extra cycles}.

In a different direction, thanks to the classification of
Mumford--Tate groups of low-dimension abelian varieties (see
\cite[\S2]{Moonen-Zarhin-Hg-cls-av-low-dim} for instance), one is able
to analyze the equivalent conditions in Theorem~\ref{T:A abs simp} for each pair of realizable
Mumford--Tate groups. The result of this study is  Theorem~\ref{Thm:
  torsion_finite_low_dim}. In particular, this theorem allows the
comparison of Type IV varieties, which is not covered by the result of
\cite{Lombardo-Hg-gp-pair-AV-2016}. Details are given in Section
\ref{Sect: low dim}.

\subsection{CM abelian varieties}
\label{Sect: CM}

\subsubsection{CM types and Mumford--Tate groups}

We start by briefly reviewing some background on CM abelian
varieties. See \cite{MilCM} and
\cite{Serre-Tate-good-red-AV} for more details.

Let $A$ be a $g$-dimensional abelian variety over a number field
$K$, and let $E$ be a CM-algebra of dimension $2g$. We say $A$ has
complex multiplication by $E$ if there exists an embedding of algebras $i:E\hookrightarrow \End^{0}(A)\coloneqq\End (A)\otimes \Q$. In this case, we call $A$ a CM abelian variety and say that $A$ has CM by $E$.  (If there is some finite extension $L/K$ such that $A_L$ has CM by $E$, we will say that $A$ has potential CM by $E$.)

The embedding $i$
induces an $E$-action on the Lie algebra $\Lie(A_{\C})$. The character
of this $E$-representation is given by $\sum_{\varphi\in \Phi}\varphi$
for some subset $\Phi\subset \Hom_{\Q} (E, \C)$. Let $c$ be the
complex conjugation on $\C$. Then
\begin{equation}\label{Eqn: half_set}
    \Phi\sqcup \Phi^{c}=\Hom_{\Q}(E, \C),
\end{equation}
where $\Phi^{c} =\{c\circ \varphi|\varphi\in \Phi\}$. The pair
$(E, \Phi)$ is called the \emph{CM-type} of $A$ (or $(A, i)$). On the other hand,
for any pair $(E, \Phi)$ where $E$ is a CM-algebra and the subset
$\Phi\subset \Hom_{\Q}(E, \C)$ satisfies \eqref{Eqn: half_set}, there
exists a CM abelian variety $A_0$ defined over a number field equipped with an action $i:E \hookrightarrow \End^0(A_0)$ with CM-type $(E, \Phi)$. The correspondence between CM-types and isogeny classes of abelian varieties with action by a CM algebra is well-understood; see, e.g., \cite[Proposition~3.12]{MilCM}.

Fix an embedding $K\hookrightarrow \cx$, and let $\bar K$ be the
algebraic closure of $K$ in $\cx$. Then any embedding
$E \hookrightarrow \cx$ factors through $\bar K$, and we have a
bijection $\Hom_{\Q}(E, \overline{K})\stackrel\sim\to\Hom_{\Q}(E,
\C)$. Let $\til{E}$ be a Galois extension of $\Q$ in $\overline{K}$
which splits $E$. Then the Galois group $\gal(\til{E}/\Q)$ acts on
$\Hom_{\Q}(E, \overline{K})$ by left composition. Let $H = H_{(E,\Phi)}$ be the group
$\{\sigma\in \gal(\til{E}/\Q)| \Phi^{\sigma}=\Phi\}$. The fixed field
$E^{\ast}:=\til{E}^{H}$ of $H$ is called the reflex field of the CM-type $(E, \Phi)$.

We now introduce the reflex norm
associated with the CM abelian variety $A$. Recall (\S \ref{Sect: Goursat's_lemma}) that for any finite extension
$F$ of $\Q$, we let $T^{F}:=\Res_{F/\rat}\mathbb{G}_{m}$ be the Weil restriction of the multiplicative group. We define
$\widetilde{\Phi}:=\{\widetilde{\varphi}\in \Hom_{\Q}(\widetilde{E},
\overline{K})\mid \widetilde{\varphi}|_{E}\in \Phi\}$. Since $\widetilde E/\Q$ is Galois, 
$\Hom_{\Q}(\widetilde{E}, \overline{K})$ is a torsor under 
$\Gal(\widetilde{E}/\Q)$, and the choice of embedding $\widetilde E
\hookrightarrow \bar K$ gives a bijection between these two sets. We
use this to define, for $\phi \in \Hom_{\Q}(\widetilde E,
\bar K)$, the (group-theoretic) inverse $\phi\inv \in
\Hom_{\Q}(\widetilde E, \bar K)$.

Let $\widetilde{\Phi}^{-1}$ be the set
$\{\widetilde{\varphi}^{-1}\mid \widetilde{\varphi}\in
\widetilde{\Phi}\}$. The map
$N_{\til{\Phi}^{-1}}: \til{E}^{\ast}\rightarrow \til{E}^{\ast}$ given
by $N_{\til{\Phi}^{-1}}(a)=\prod_{\sigma\in \til{\Phi}^{-1}}\sigma(a)$
defines a map of algebraic tori
$N_{\til{\Phi}^{-1}}: T^{\til{E}}\rightarrow T^{\til{E}}.$ This map
factors through $T^{E^{\ast}}$ and has image contained in $T^{E}$. More precisely,
we have a commutative diagram of $\Q$-tori
\[
\begin{tikzcd}
 T^{\til{E}}\ar[r, "N_{\til{\Phi}^{-1}}"]\ar[d, "N_{\til{E}/E^{\ast}}"'] & T^{\til{E}} \\
 T^{E^{\ast}} \ar[r, "N_{\Phi}"] & T^{E}\ar[u, hook]
\end{tikzcd},
\]
where $N_{\til{E}/E^{\ast}}$ is the usual norm map and $N_{\Phi}$ is
called the reflex norm of $(E, \Phi)$. For any finite extension
$L/E^{\ast}$, we define $N_{L, \Phi}=N_{\Phi}\circ N_{L/E^{\ast}}$. We
will call $N_{L, \Phi}$ the $L$-reflex norm of $(E, \Phi)$. Note that
$N_{L/E^{\ast}}$ is a surjective map. The image of the map
$N_{L,\Phi}: T^{L}\rightarrow T^{E}$  is independent of
$L$, and we denote it by $T_{\Phi}$.  See, e.g., \cite[Lem.~4.2]{Yu-MTCM} for an explicit calculation of $X^*(T_\Phi)$. 
 
Consider the special case where $E$ is a Galois extension of $\rat$.  Then $E$ contains the reflex field $E^*$, and $T_{\Phi}$, as the image of $T^E$ under $N_{E,\Phi}: T^E \to T^{E^*} \to T^E$, is a \emph{quotient} of $T^E$.  Inside 
\begin{align}
    X^*(T^E) &\iso \Z \ang{\sigma | \sigma \in \gal(E/\rat)}  \end{align}
we find that
\begin{align}    
\label{Eqn: char_of_T_Phi}
    X^{\ast}(T_{\Phi})\cong & \Z\langle \sum_{\varphi\in \Phi}(\sigma\circ \varphi^{-1})\mid \sigma\in \gal(E/\Q) \rangle\subset X^{\ast}(T^{E}).
\end{align}

\begin{lemma}\label{Lem: MTC_CM_Ab.var}
If $A$ is a CM abelian variety over a number field, then the Mumford--Tate conjecture holds for $A$.  Moreover, if $(E,\Phi)$ is a CM type for $A$, then $\mt(A) \cong T_\Phi$.
\end{lemma}
\begin{proof}
This is due to Pohlmann \cite[Thm.~5]{Pohlmann}; see also \cite[Lemma~4.2]{Yu-MTCM} for a modern proof.  (In fact, while \cite{Yu-MTCM} focuses on simple abelian varieties, the proof given there works verbatim in the non-simple case, too.)
\end{proof}

\begin{remark}\label{R: replace Galois CM}
Note in particular that (the character group of) the torus $\mt(A)$ can be explicitly described using a CM-type $(E, \Phi)$ of $A$. It is usually more convenient to assume that $E$ is Galois over $\Q$; and in studying the essential torsion finiteness problem for CM abelian varieties, we can always do this. 
Indeed, if $E/\Q$ is not Galois, choose a CM Galois extension $E'/\Q$ such that $E\subseteq E'$, and let $n=[E':E]$. Then $A^n$ has a CM-type $(E', \Phi')$ where $\Phi'=\{\sigma\in \Gal(E'/\Q)\mid \sigma\mid_{E}\in \Phi  \}$.  Moreover,   $\mt(A)\cong \mt (A^n)$; and if $A$ is defined over a field $K$, then $K_A = K_{A^n}$ (Lemma \ref{L:tf for products}). Thus, for our purposes, we may restrict our attention to \emph{Galois} CM fields.
\end{remark}

\subsubsection{Galois representations}
\label{SS:CM galois rep}
For use in later examples, we recall the calculation of the Galois representation of a CM abelian variety following \cite[\S7]{Serre-Tate-good-red-AV} and \cite[\S3]{Yu-MTCM}.  Let $A/K$ be an abelian variety with complex multiplication by $E$.  Let $\ell$ be a rational prime which does not divide the index $[\mathcal O_E: \End(A)]$.  (This condition only rules out finitely many primes.  Alternatively, for our applications we may replace $K$ by a finite extension and adjust $A$ in its isogeny class, in which case we may assume that $\End(A) = \mathcal O_E$.)  Then $\mathcal O_{E_\ell} := \mathcal O_E \otimes \integ_\ell$ is a direct sum of discrete valuation rings, and the Tate module $T_\ell A$ is free of rank one over $\mathcal O_{E_\ell}$.  The Galois group of $K$ acts $E$-linearly on $T_\ell A$, and so 
the Galois representation $\rho_{A/K,\ell^\infty}$ of $\gal(K)$ factors through $\gal(K)^\ab$.  Composing the $\ell$-adic representation with the Artin reciprocity map (in the idelic formulation of class field theory), one obtains a continuous group homomorphism which we still denote by $\rho_{A/K,\ell^\infty}$: \[
\xymatrix{
\mathbb A^\times_K \ar[r]^-{\widetilde \rho_{A/K,\ell^\infty}} \ar[d]_{\operatorname{art}} & \mathcal O_{E_\ell}\units = T^E(\integ_\ell) \subset \GL(T_\ell A).\\
\gal(K)^{\operatorname{ab}} \ar[ur]_{\rho_{A/K,\ell^\infty}}&
}
\]
After possibly replacing $K$ with a finite extension, we now assume that $K$ contains $E^*$, the reflex field of the CM-type of $A$, so that the reflex norm $N_{K,\Phi}$ is defined (\S \ref{Sect: CM}).  
Then by \cite[Theorem~10, 11]{Serre-Tate-good-red-AV} we can concretely describe the representation $\widetilde\rho_{A/K,\ell^\infty}$ by
\begin{equation}
\label{Eqn: Serre-Tate}
\xymatrix{
\mathbb A_K\units \ar[r]^-{\widetilde\rho_{A/K,\ell^\infty}} & T^E(\integ_\ell) = \mathcal O_{E_\ell}\units , \\
a = (a_v)_v \ar@{|->}[r] & \varepsilon(a) N_{K,\Phi,\ell}(a_\ell\inv).
}
\end{equation}
Here $a_\ell=(a_v)_{v\mid \ell}$ denotes the component of $a$ in $K_{\ell}\units=\prod_{v\mid \ell}K_v\units$; and the map 
\[
\xymatrix{
N_{K, \Phi, \ell}= N_{K, \Phi} \times \Q_{\ell}: K_{\ell}^{\times}=(K\otimes \Q_{\ell})^{\times}\ar[r] &  E_{\ell}^{\times}=(E\otimes \Q_{\ell}})^{\times}
\]
is induced by the reflex norm map from $T^K$ to $T^E$; and 
\[\xymatrix{
\varepsilon:  \mathbb{A}_K^{\times} \ar[r] & E^{\times}}
\]
is the unique homomorphism satisfying the following conditions: 
\begin{enumerate}[label=(\alph*)]
    \item The restriction of $\varepsilon$ to $K^{\times}$ is the reflex norm map $N_{K, \Phi}: K^{\times }\to E^{\times}$.
    \item The homomorphism $\varepsilon$ is continuous, in the sense that its kernel is open in $ \mathbb{A}_K^{\times}$.
    \item There is a finite set $S$ of places of $K$, including the infinite ones and those where $A$ has bad reduction, such that 
    \[
    \varepsilon(a)=\prod_{v\notin S} \pi_{v}^{\nu(a_v)} \quad \text{ for all }a=(a_v) \text{ with }a_v=1 \text{ when }v\in S
    \]
    where each $\pi_v$ is the Frobenius element attached to $v$ \cite[p.511]{Serre-Tate-good-red-AV}. 
\end{enumerate}

\subsubsection{Nondegenerate abelian varieties}
\label{S:nondegenerate}
 
Let $A/K$ be an abelian variety with CM type $(E,\Phi)$.  Recall that $A$ is called nondegenerate if $\dim \mt(A)$ is maximal, i.e., $\dim \mt(A) = \dim A+1$. Recall that $\mathcal G_{A/K,\rat_\ell} $ is the Zariski closure of the $\ell$-adic representation image (c.f. \S~\ref{Sect: connectedness}).

\begin{lemma}
\label{L:MT nondegenerate}
    Let $A/K$ be an abelian variety with nondegenerate CM type $(E,\Phi)$.  Then for each $\ell$,  $\mathcal G_{A/K,\rat_\ell} = \mt(A)\times_\rat \rat_\ell$, i.e., $K = K^{\text{conn},A}$.
\end{lemma}

\begin{proof}
By Lemma \ref{Lem: K^conn}, it suffices to prove the statement for a single $\ell$, and so we assume that $\ell \nmid [\mathcal O_E: \End(A)]$.
    As we have seen above, $T_\ell A$ admits commuting actions by $\mathcal O_{E_\ell}$ and $\gal(K)$, and thus the $\ell$-adic representation $\gal(K) \to \aut(T_\ell A)$ factors through $\mathcal O_{E_\ell}\units$.  A choice of $K$-rational polarization on $A$ induces a symplectic form $\psi$ on $T_\ell A$, which is also preserved by $\gal(K)$ up to a scaling.  Thus, the image of $\gal(K)$ in $\aut(T_\ell A)$ is contained in $\mathcal O_{E_\ell}\units \cap \operatorname{GSp}(T_\ell A, \psi)$.  Since these are the $\integ_\ell$-points of a maximal torus in $\operatorname{GSp}_{2\dim A}$ -- indeed, both $\rank \operatorname{GSp}_{2\dim A}$ and $\dim \mt(A)$ are $1+\dim A$ --  the result follows.
\end{proof}

\subsubsection{Torsion finiteness}

With this preparation, we can now use Theorem \ref{T:A abs simp} to characterize the essential torsion finiteness of CM abelian varieties in terms of CM types.

Recall that an abelian variety $A$ over a field $K$ is called \emph{isotypic} if it is isogenous to a power of a simple abelian variety over the same field $K$, i.e., up to isogeny, $A$ has a unique simple factor \cite[Defn. 1.2.5.2]{chai-conrad-oort}. Any CM abelian variety is isogenous to a product of isotypic CM abelian varieties \cite[Prop.~1.3.2.1]{chai-conrad-oort}, and an isotypic CM abelian variety is geometrically isotypic \cite[Cor.~1.3.7.2]{chai-conrad-oort}. 
 
To state our next result, it is more convenient to name our abelian varieties $A_1$ and $A_2$.    In this case we will change notation slightly and write, for example, $G_1$ and $H_{12}$ for $G_{A_1}$ and $H_{A_1,A_2}$.

\begin{thm}
\label{Thm: main_thm_CM_pair}
    Let $A_1$ and $A_2$ be two isotypic  abelian varieties over a number field $K$, with $A_i$ of potential CM-type $(E_i, \Phi_i)$.  Let $T_i = T_{\Phi_i} = \mt(A_i)$ be the Mumford--Tate group of $A_i$, and let $T_{12} = \mt(A_1\times A_2)$.  Use the surjection $T_{12} \twoheadrightarrow T_i$ to identify $X^*(T_i)$ with a submodule of $X^*(T_{12})$.  Then either:
    \begin{enumerate}[label=(\alph*)]
    \item $X^*(T_1)\otimes \rat \subseteq X^*(T_2)\otimes \rat$.  Then $A_1$ is potentially torsion infinite for $A_2$.  For each $\ell$, there exists a finite extension $\widetilde K_\ell/K$ such that 
    \[
    A_1[\ell^\infty]((\widetilde K_\ell)_{A_2,\ell^\infty}) = A_{1,\ell^\infty}.
    \]
    Moreover,  if $X^*(T_1)\subseteq X^*(T_2)$, and if $A_1$ is simple and nondegenerate, then      \[
    A_1(K_{A_2})_{\tors} = A_1(\bar K)_{\tors}.
    \]
    \item $X^*(T_1)\otimes\rat\not\subseteq X^*(T_2)\otimes\rat$.  Then $A_1$ is essentially torsion finite for $A_2$.
    \end{enumerate}
\end{thm}

\begin{proof}
By Theorem \ref{T:A abs simp} and Remark \ref{R:A isotypic}, $A_1$ is potentially torsion infinite for $A_2$ if and only if $\dim H_{12} = 0$.  Since $H_{12}$ is of multiplicative type, this happens if and only if $\dim_{\Q} X^*(H_{12})\otimes \Q = 0$.  After tensoring both sides of \eqref{E:HAB for tori reprise} with $\Q$, we find that this happens if and only if $X^*(T_1)\otimes \Q \subseteq X^*(T_2)\otimes\Q$.

If $A_1$ is potentially torsion infinite for $A_2$, then the description of $\widetilde K_\ell$, etc. is in Lemma \ref{L:newmainHabarbitrary}.  

Finally, suppose we have an inclusion of integral lattices $X^*(T_1)\subseteq X^*(T_2)$ and that $A_1$ is simple and nondegenerate.  The calculation \eqref{E:HAB for tori reprise} shows that $H_{12}$ is trivial.  Briefly suppose that $A_1/K$ has CM actually defined over $K$, and thus (Lemma \ref{L:MT nondegenerate}) has connected Galois representations.  Then for each natural number $N$ we have a containment $\Gamma_{A_1,N} \subset T_1(\integ/N)$, and thus $\gal(K_{A_2})$ acts trivially on $A_{1,N}$.

Now suppose that $A_1/K$ merely has potential complex multiplication.  The surjection $T_2 \to T_1$ means that the splitting field of $T_2$ contains the splitting field of $T_1$; equivalently, we have an inclusion of reflex fields $E_1^* \subset E_2^*$.  Suppose $N\ge 3$ is an integer.  Then all geometric endomorphisms of $A_2$ are defined over $K_{A_2,N}$ \cite{silverberg-FoD}.  Therefore $K_{A_2,N}$  contains $E_2^*$, and thus $E_1^*$ \cite[Prop.~7.11]{MilCM}.  Because $E_1^*$ is simple, all geometric endomorphisms of $A_1$ are defined over $K_{A_2,N}$.  Therefore, the image of the action of  $\gal(K_{A_2,N})$ on $A_{1,N}$, and \emph{a fortiori} that of $\gal(K_{A_2})$, is contained in $T_1(\integ/N)$ (Lemma \ref{L:MT nondegenerate}), and we conclude as before.

     \end{proof}

\begin{remark}
\label{R:CM common field}
    In the context of Theorem \ref{Thm: main_thm_CM_pair},  let $E/\rat$ be a Galois CM field containing $E_1$ and $E_2$, and assume that $A_1$ and $A_2$ share no common geometric isogeny factor.  As in Remark \ref{R: replace Galois CM}, after replacing $A_1$ and $A_2$ by suitable powers, we may assume $A_1$ and $A_2$ have CM by the same field $E$, with respective CM types $\Phi_{E,1}$ and $\Phi_{E,2}$; then $A_1\times A_2$ has a CM type $(E\times E, \Phi_{12})$, where $\Phi_{12} = \Phi_{E,1}\sqcup \Phi_{E,2}$.  Then $T_{\Phi_i} = T_{E,\Phi_i}$, and
the compatibility of the various (reflex) norm maps is expressed in the commutativity of the following diagram of tori:
\[
\begin{tikzcd}
      &                    & T_{\Phi_1}     \\
T^{E}\ar[urr, bend left=15, two heads, "N_{E, \Phi_1}"]\ar[drr, bend right=15, two heads, "N_{E, \Phi_2}"']\ar[r, two heads, "N_{E, \Phi_{12}}"] & T_{\Phi_{12}}\ar[ur, two heads, "\pi_1"']\ar[dr, two heads, "\pi_2"] &    \\
      &                    & T_{\Phi_2}   \\
\end{tikzcd}.
\]
In particular, in Theorem \ref{Thm: main_thm_CM_pair}, we may compare $X^*(T_{\Phi_1})$ and $X^*(T_{\Phi_2})$ inside $X^*(T^E)$ (or $X^{\ast}(T^{E})\otimes \Q$).
\end{remark}

\subsubsection{Examples}

In concrete cases, Theorem \ref{Thm: main_thm_CM_pair} gives a way to explicitly analyze essential torsion finiteness for pairs of CM abelian varieties.

\begin{exa}\label{eg: zeta_13}
Let $E = \Q(\zeta_{13})$. Then $\Gal(E/\Q)\cong\langle \sigma|
\sigma^{12}=1\rangle$. There are exactly six isomorphism classes of CM-types for $E$, with representatives
    \begin{align*}
\Phi_1&=\{1, \sigma, \sigma^2,
\sigma^3, \sigma^4, \sigma^5\}, &
\Phi_4 &= \{ 1, \sigma^7, \sigma^8, \sigma^3, \sigma^4, \sigma^5\},\\
\Phi_2&=\{1, \sigma^7, \sigma^2,
\sigma^3, \sigma^4, \sigma^5\},
& \Phi_5 &= \{1, \sigma^7, \sigma^8, \sigma^3, \sigma^{10}, \sigma^5\}, \\
\Phi_3&=\{1, \sigma, \sigma^8,
\sigma^3, \sigma^4, \sigma^5\},&
\Phi_6&=\{1, \sigma^4, \sigma^8,
\sigma, \sigma^5, \sigma^9\}.
\end{align*}
Let $A_i$ be an abelian sixfold with CM-type $(E,\Phi_i)$.  Then any abelian variety with CM by $E$ is
geometrically isogenous to one of the $A_i$; and for $1 \le i \le 5$,
$A_i$ is geometrically simple.
By an explicit computation, one can check that
in $X^{\ast}(T^{E})\otimes \Q$ we have  
\begin{align*}
    X^{\ast}(T_{\Phi_{i}})\otimes \Q & =X^{\ast}(T_{\Phi_{j}})\otimes
                                       \Q, \quad 1 \le i,j\leq 5 \text{, and }\\
    X^{\ast}(T_{\Phi_{6}})\otimes \Q& \subsetneq X^{\ast}(T_{\Phi_{i}})\otimes \Q, \quad 1\leq i\leq 5.
\end{align*}
Then for any $i
\leq 5$, $A_i$ is essentially torsion finite for $A_6$; and for any
$j \le 6$, $A_j$ is potentially torsion infinite for $A_i$.

We also can compute $X^{\ast}(H_{ij})$ explicitly. 
For example, consider $A_1$ and $A_2$. By \eqref{Eqn: char_of_T_Phi}, $X^{\ast}(T_{\Phi_1})$ is generated by the Galois orbit of $ 1+\sigma^{-1}+\sigma^{-2}+\sigma^{-3}+\sigma^{-4}+\sigma^{-5}$ and $X^{\ast}(T_{\Phi_2})$ is generated by the Galois orbit of $1+\sigma^{-7}+\sigma^{-2}+\sigma^{-3}+\sigma^{-4}+\sigma^{-5}$. Note that the Galois orbit of $ 1+\sigma^{-1}+\sigma^{-2}+\sigma^{-3}+\sigma^{-4}+\sigma^{-5}$ (resp. $1+\sigma^{-7}+\sigma^{-2}+\sigma^{-3}+\sigma^{-4}+\sigma^{-5}$) equals the Galois orbit of  $1+\sigma^{1}+\sigma^{2}+\sigma^{3}+\sigma^{4}+\sigma^{5}$ (resp. $1+\sigma^{7}+\sigma^{2}+\sigma^{3}+\sigma^{4}+\sigma^{5}$). We compute
\begin{equation}
\label{E:sample CM torus calc}
\begin{split}
1+\sigma^7+\sigma^2+\sigma^3+\sigma^4+\sigma^5=&(1+\sigma+\sigma^2+\sigma^3+\sigma^4+\sigma^5)-(\sigma+\sigma^2+\sigma^3+\sigma^4+\sigma^5+\sigma^6)\\ &+(\sigma^2+\sigma^3+\sigma^4+\sigma^5+\sigma^6+\sigma^7) \\
=& (1+\sigma+\sigma^2+\sigma^3+\sigma^4+\sigma^5)-\sigma (1+\sigma+\sigma^2+\sigma^3+\sigma^4+\sigma^5) \\ &+ \sigma^2(1+\sigma+\sigma^2+\sigma^3+\sigma^4+\sigma^5).
\end{split}
\end{equation}
Then $X^{\ast}(T_{\Phi_{2}})\subseteq X^{\ast}(T_{\Phi_{1}})$. So, by \eqref{E:HAB for tori reprise} \[
X^{\ast}(H_{21})\cong \frac{X^{\ast}(T_{\Phi_2})}{X^{\ast}(T_{\Phi_2})\cap X^{\ast}(T_{\Phi_1})}\cong \frac{X^{\ast}(T_{\Phi_2})}{X^{\ast}(T_{\Phi_2})}\cong 0.
\]
The $A_1$ and $A_2$ are primitive. Then $X^{\ast}(T_{\Phi_1})$ is a rank 7 free $\Z$-module with a basis $\{\sigma^{i} + \sigma^{i+1} + \sigma^{i+2} + \sigma^{i+3} + \sigma^{i+4} +\sigma^{i+5}\mid 0\leq i \leq 6\}$ and the $X^{\ast}(T_{\Phi_2})$ is a rank 7 free $\Z$-module with a basis $\{\sigma^{i} + \sigma^{i+7} + \sigma^{i+2} + \sigma^{i+3} + \sigma^{i+4} +\sigma^{i+5}\mid 0\leq i \leq 6\}$. A more detailed linear algebra calculation shows that 
\[
X^{\ast}(H_{12})\cong \frac{X^{\ast}(T_{\Phi_1})}{X^{\ast}(T_{\Phi_2})\cap X^{\ast}(T_{\Phi_1})}\cong \frac{X^{\ast}(T_{\Phi_1})}{X^{\ast}(T_{\Phi_2})}\cong (\Z/2)^{2}.
\]
\end{exa}

In Lemma \ref{L:newmainHabarbitrary}, one might hope that $\widetilde{K}_{\ell}$ could be chosen independently of $\ell$. To end this section, we will explain this is impossible in general by considering the following example.

\begin{exa}\label{Exa: torsion_infinite-11} 

    Let $E$ be $\Q(\zeta_{11})$. Then $\Gal(E/\Q)\cong\langle \sigma| \sigma^{10}=1\rangle$. There are exactly four isomorphism classes of CM-type for $E$, with representatives
    \begin{align*}         
\Phi_1&=\{1, \sigma^2, \sigma^4,
\sigma^6, \sigma^8\}, &
\Phi_3 &= \{1, \sigma^3, \sigma^6, \sigma^9, \sigma^2\},\\
\Phi_2 &= \{ 1, \sigma^6, \sigma^2, \sigma^3, \sigma^4\},
& \Phi_4&=\{1, \sigma, \sigma^2,
\sigma^3, \sigma^4\}. \\
\end{align*}
Let $K$ be a number field containing $E$ (and, in particular, the reflex fields of each $\Phi_i$), and let $A_{i}/K$ be an abelian fivefold with CM-type $(E, \Phi_{i})$. Further assume that each $A_i$ has independent representations over $K$, and that each $A_i$ has everywhere good reduction.  Let $S$ be a sufficiently large finite set of primes of $K$ so that the description of the Galois representations in \S \ref{SS:CM galois rep} holds.    Then any abelian variety with CM by $E$ is geometrically isogenous to one of the $A_i$.  For $2\leq i \leq 4$, $A_i$ is geometrically simple, while $A_1$ is geometrically isogenous to the cube of an elliptic curve with complex multiplication by $\rat(\sqrt{-11})$. By an explicit computation, one can check that in $X^{\ast}(T^{E})\otimes \Q$ we have
\begin{align*}
    X^{\ast}(T_{\Phi_{i}})\otimes \Q & =X^{\ast}(T_{\Phi_{j}})\otimes \Q, \quad 2 \le i,j\leq 4 \text{, and }\\
    X^{\ast}(T_{\Phi_{1}})\otimes \Q& \subsetneq X^{\ast}(T_{\Phi_{i}})\otimes \Q, \quad 2\leq i \leq 4.
\end{align*}
For any $i\geq 2$, $A_i$ is essentially torsion finite 
 for $A_1$; and for any $j\leq 4$, $A_{j}$ is potentially torsion infinite for $A_{i}$.
     
Now we focus on $A_1$ and $A_2$. 
 
Identifying $X^*(T^E)$ with the group ring $\integ[\ang \sigma]$, we may present $X^*(T_{\Phi_1})$ as $X^*(T_{\Phi_1}) \iso \integ\oplus \integ$, with basis elements $\sum_{0 \le j \le 5} \sigma^{2j}$ and $\sum_{0 \le j \le 5} \sigma^{2j+1}$.  The action of the generator $\sigma$ of $\gal(E/\rat)$ is to exchange these two basis vectors.  In particular, the torus is split by the fixed field $\rat(\zeta_{11})^{\sigma^5} = \rat(\sqrt{-11})$.

An explicit computation shows that
\[
X^*(T_{\Phi_1})\cap X^*(T_{\Phi_2}) = 3 X^*(T_{\Phi_1}).
\]
In particular, $X^*(H_{12}) = X^*(T_{\Phi_1})/3 X^*(T_{\Phi_1})$.  Thus ${H_{12}}_{\bar\rat} \iso (\integ/3)_{\bar\rat} \oplus (\integ/3)_{\bar\rat}$, and the action of $\gal(\rat(\sqrt{-11})/\rat)$ exchanges the two components.  

On one hand, $H_{12}$ is zero-dimensional.  Consequently, by Theorem \ref{T:A abs simp} (and the following remark), $A_1$ is potentially torsion infinite for $A_2$.

On the other hand, $H_{12}$ is not split over $\rat$, although it does admit $\rat$-points.  Indeed, $H_{12}(\rat) = \set{(0,0),(1,1),(2,2)} \subsetneq (\integ/3)\oplus (\integ/3)$.  Consequently, Lemma \ref{L:newmainHabarbitrary} only implies that, for each $\ell$, there exists some finite extension $\widetilde K_{2,\ell^\infty}$ of $K_{2,\ell^\infty}$ such that $A_1[\ell^\infty](\widetilde K_{2,\ell^\infty})$ is infinite.   We will now use the explicit calculation of the action of Galois to show that for any finite extension $L/K$, $A_1[\ell^\infty](L_{2,\ell^\infty})$ is finite for all but finitely many primes $\ell$.  This shows  that Lemma \ref{L:newmainHabarbitrary} is essentially optimal.

Let $\rho_{1 , \ell}$ and $\rho_{2, \ell}$ denote $\rho_{A_1/K,\ell^\infty}$ and $\rho_{A_2/K,\ell^\infty}$, respectively; and let $\widetilde \rho_{i,\ell}$ be the pullback of $\rho_{i,\ell}$ to $\mathbb A_K\units$.

Now suppose $\ell\not \in S$, 
 and embed $K_\ell^{\times}=\prod_{v\mid \ell} K_{v}^{\times}$ naturally into the $\ell$-adic component of $ \mathbb{A}_K^{\times}$.  Then the restriction of the $\ell$-adic representations \eqref{Eqn: Serre-Tate} to $\OO_{K_{\ell}}^{\times}$ reads as 
\begin{equation}\label{Eqn: N_K_simple_ver}
    \widetilde\rho_{i,\ell}(a)= N_{K, \Phi_i}(a_{\ell}^{-1})\quad \text{ if }a\in \OO_{K_\ell}^{\times} \text{, and }i=1,2.
\end{equation}
Now further assume that $\ell$ is a prime integer that totally splits in $K$, and use the fact that the reflex norm $N_{K, \Phi_i}=N_{E, \Phi_i}\circ { N}_{K/E}$, where $N_{K/E}$ is the usual norm map of fields. Then $K_{\ell}$ is unramified over $E_{\ell}$, and thus  $N_{K/E}(\OO_{K_{\ell}}^{\times})=\OO_{E_{\ell}}^{\times}$ \cite[V. \S2, Corollary of Proposition~3]{Serre-LF}. Hence \eqref{Eqn: N_K_simple_ver} factors through (which will still be denoted by $\widetilde\rho_{i, \ell}$ to ease notation)
\[
\widetilde\rho_{i,\ell}(a)= N_{E, \Phi_i}(a_{\ell}^{-1})\quad \text{ if }a\in \OO_{E_\ell}^{\times} \text{, and }i=1,2.
\]
Since $\ell$ splits in $K$, it is also totally split in the sub-extension $E$. Recall that $\Gal(E/\Q)=\langle \sigma \rangle\simeq \Z/(10)$, hence 
\[
\OO_{E_\ell}^{\times}=(\Z_\ell\otimes E)^{\times}\simeq \prod_{\tau\in \Gal(E/\Q) }\Z_{\ell,\tau}^{\times}=\prod_{i=0}^{9} \Z_{\ell, \sigma^i}^{\times},
\]
With respect to this isomorphism, every element $a_{\ell}$ of $\OO_{E_\ell}^{\times}$ can be expressed by a vector of the form $a_{\ell}=(a_1, a_2, a_3,\cdots, a_{10})$, and $\sigma\in \Gal(E/\Q)$ acts on $a_{\ell}$ by cyclically permuting its coordinates. By definition, 
\[
N_{E,\Phi_i}(a_\ell) = \prod_{\tau \in \Phi_i}\tau\inv(a_\ell).
\]
Fix $x,y \in \integ_\ell\units$, and consider the element $a_\ell \in \mathcal O_{E_\ell}\units$ with coordinates
\begin{align*}
a_\ell &= (x^{-2}y^{-1}, x^{-1}y^{-1}, xy, x^2y^2, x^2y, y, 1,1,1,x).
\intertext{Direct computation then shows that}
N_{E,\Phi_1}(a_\ell) &= (x y, x^{2} y^{2}, x y, x^{2} y^{2}, x y, x^{2} y^{2}, x y, x^{2} y^{2}, x y, x^{2} y^{2})\\
N_{E,\Phi_2}(a_\ell) &= (x^{3} y^{3}, x^{3} y^{3}, x^{3} y^{3}, x^{3} y^{3}, 1, 1, 1, 1, 1, x^{3} y^{3}) 
\end{align*}
Now suppose that $xy$ is a primitive third root of unity; this is possible, of course, exactly if $\ell \equiv 1 \bmod 3$.  On one hand, $N_{E,\Phi_1}(a_\ell) = 1$; on the other hand, because $xy\not\equiv 1 \bmod \ell$ and $(xy)^2\not\equiv 1 \bmod \ell$, $N_{E,\Phi_2}(a_\ell)$ acts without fixed points on $\mathcal O_E\otimes \integ/\ell$.
In particular, $a_\ell \in\ker \widetilde\rho_{2,\ell}\smallsetminus \ker \widetilde\rho_{1,\ell}$. 
 Let $g\in \gal(K)$ be the image of $a_\ell$ under the reciprocity map.  
By \eqref{Eqn: N_K_simple_ver} and the formula of Serre-Tate, $g\in \ker \rho_{2,\ell}\smallsetminus \ker\rho_{1,\ell}$. Moreover, since $\rho_{1,\ell}(g)= \widetilde\rho_{1,\ell}(a_{\ell})\in E^{\times}$, we know that it does not have eigenvalue $1$, even when working with $\integ/\ell$-coefficients. In particular, $A_1[\ell](K_{A_2,\ell^\infty})$ is trivial.

Finally, notice that we can produce such an $a_\ell$ for each $\ell$ with $\ell \equiv 1 \bmod 3$ which is totally split in $K$.  Since $A_1$ has independent extensions, there is not a finite extension $L/K$ on which each $\operatorname{art}(a_\ell)$ acts trivially.  In particular, there is no finite extension $L/K$ such that, for each $\ell$, $A_1[\ell^\infty](L_{A_2,\ell^\infty}) = A_1[\ell^\infty](\bar L)$.

\end{exa}

\begin{exa}     In \cite[Thm.~5.1]{Lombardo-iso-Kummerian}, Lombardo gives a construction of an infinite family of iso-Kummerian CM pairs of abelian varieties.  We briefly interpret his work in the framework developed here.

    Given an a CM field $E$ which is the compositum of a cyclic totally real field of dimension $g$ and a quadratic imaginary field, and the auxiliary choice of two integers $r$ and $h$, Lombardo defines two different CM types $\Phi_1$ and $\Phi_2$, and chooses corresponding abelian varieties $A_1$ and $A_2$.  After passage to a suitably large common field of definition $K$, one shows that the kernels of the $\ell$-adic representations $\widetilde \rho_{A_k,K,\ell^\infty}$ coincide.

This calculation shows that the characters of $T^E$ which vanish on the image of $T^K$ under $N_{K,\Phi_1}$ are the same as those characters which vanish on the image of $T^K$ under $N_{K,\Phi_2}$.  Consequently, $\mt(A_1)$ and $\mt(A_2)$ are the same sub-torus of $T_E$, and thus $A_1$ and $A_2$ are mutually torsion infinite.
    
\end{exa}

 \subsection{Extra Hodge classes and torsion infiniteness}\label{Sect: extra cycles}
Following the notation in the previous section, let $A_1$ and $A_2$ be two isotypic CM abelian varieties over $K$.  In this section, we will see that if $A_1$ is potentially torsion infinite for $A_2$, then this is explained by a certain sort of Hodge class in some degree $2w$ on some product $A_1^m \times A_2^n$.  The particular values of $w$, $m$ and $n$ are not unique; and even once these are specified, the class itself, or even its $\rat$-span is not canonical.  Consequently, we will call any such Hodge class a ''torsion-infinite Hodge class from $A_1$ to $A_2$'', even though it actually lives on some unspecified product $A_1^m \times A_2^n$.

Suppose that $A_i$ has a CM-type $(E, \Phi_i)$ where $E$ is a CM Galois extension of $\Q$. We also assume that the base field $K$ is sufficiently large (e.g., it contains $E$).

We first describe the Hodge classes on $A_{1}^{m}\times A_{2}^{n}$ (see \cite{Pohlmann} for more details). Let  $V_i=H^{1}(A_{i}, \Q)$.  Recall that $\ang{V_i}$, the tensor category generated by $V_i$,  is equivalent to the category $\Rep_{\Q}(\mt(A_i))$ of representations of $\mt(A_i)$. By Lemma \ref{Lem: MTC_CM_Ab.var} and our assumption for $E$, the reflex norm defines a quotient map $N_{E, \Phi_i}: T^{E}\twoheadrightarrow \mt(A_{i})$, which induces a fully faithful map on the categories of representations $\Rep_{\Q}(\mt(A_i))\rightarrow \Rep_{\Q}(T^{E})$. This allows us to describe the Hodge classes on $A_{1}^{m}\times A_{2}^{n}$ using the representation theory of the algebraic torus $T^{E}$ for any positive integers $m$ and $n$.

We denote the representation $T^{E}\twoheadrightarrow \mt(A_i) \hookrightarrow \GL_{V_i}$ by $\rho_i$. Note that $X^{\ast}(T^{E})\cong \bigoplus_{\sigma\in \Gal(E/\Q)} \Z\langle \sigma \rangle$, and the Galois group $\Gal (E/\Q)$ acts on it by left multiplication. Since $E/\Q$ is Galois, we can identify $\Phi_{i}$ with a subset of $\Gal(E/\Q)$. 

For any representation $\rho:T^{E}\rightarrow \GL_V$, we let $\Xi_{V}$ (or $\Xi_{\rho}$) be the collection of weights of this representation. The set $\Xi_{V}$ is a finite sub-multiset of $X^{\ast}(T^{E})$; the support $\supp(\Xi_V)$ of $\Xi_V$ -- that is, those elements of $X^*(T^E)$ with nonzero multiplicity -- is finite, and all multiplicities are finite. For future use, we note that if $\Xi_V = \set{\alpha_1, \dots, \alpha_d}$ is a set of distinct characters, then $\operatorname{supp}(\Xi_{V^{\oplus m}})$ is the same set, and each weight now occurs with multiplicity $m$.  Moreover, the support of $\Xi_{\wedge^r(V^{\oplus m})}$ is then
\begin{equation}
    \label{E:char of wedge}
\operatorname{supp}(\Xi_{\wedge^r(V^{\oplus m})}) = \set{ \sum e_i \alpha_i: \sum e_i = r\text{ and }0 \le e_i \le m\text{ for each }i}.
\end{equation}

By the definition of the reflex norm,
\[
\Xi_{V_{i}}=\set{ \sum_{\sigma \in \Phi_{i}} (\sigma\circ \varphi^{-1})\mid \varphi\in \Gal(E/\Q)}.
\]
 Since $\Gal(E/\Q)$ acts transitively on this set, we have 
\[
X^{\ast}(T_{\Phi_i})\otimes \Q\cong \sum_{\alpha \in \Xi_{V_{i}}}\Q \langle\alpha\rangle\subset X^{\ast}(T^{E})\otimes \Q
\]
as Galois modules. We also denote the 1-dimensional representation $\Nm :T^{E}\rightarrow \G_{m}$ by $\Q(1)$. The weight of the representation $\rat(1)$ is $\chi:=\sum_{\sigma\in \Gal(E/\Q)}\sigma \in X^{\ast}(T^{E})$. If $n\ge 0$ let $\rat(n) = \rat(1)^{\otimes n}$, and if $n < 0$ set $\rat(n) = \rat(1)^{\vee, \otimes -n}$.  Finally, let $\rat(0)$ denote the trivial representation of $T^E$.  If $V$ is any representation of $T^E$, then we let $V(n) = V\otimes \rat(n)$.
If $V=H^{1}(A, \Q)$ for some abelian variety $A$, then $V^{\vee}\cong V(1)$.

If $V$ is any Hodge structure, the group of Hodge classes in $V$ is $\Hom_{\Q-\mathrm{HS}}(\one, V)$, where $\one$ is the trivial Hodge structure.
 For an integer $w\ge 0$ we have
\[
\begin{split}
\Hom_{\Q-\mathrm{HS}}(\one, H^{2w}(A_{1}^{m}\times A_{2}^{n}, \Q)(w))&\cong \Hom_{T^{E}} (\Q(0), H^{2w}(A_{1}^{m} \times A_{2}^{n}, \Q)(w)) \\
&\cong H^{2w}(A_{1}^{m} \times A_{2}^{n}, \Q)(w)^{T^{E}} \\
&\cong \bigoplus_{r+s=2w}(H^{r}(A_{1}^{m}, \Q) \otimes H^{s}(A_{2}^{n}, \Q))(w)^{T^{E}} \\
&\cong \bigoplus_{r+s=2w}(\wedge^{r} V_{1}^{\oplus m} \otimes (\wedge^{s}V_{2}^{\oplus n}(w))^{T^{E}}\\
&\cong \bigoplus_{r+s=2w}((\wedge^{r} V_{1}^{\vee, \oplus m})(-r) \otimes (\wedge^{s}V_{2}^{\oplus n})(w))^{T^{E}}\\
&\cong \bigoplus_{r+s=2w}(\wedge^{r} V_{1}^{\vee, \oplus m} \otimes (\wedge^{s}V_{2}^{\oplus n})(w-r))^{T^{E}}.
\end{split}
\]
So the $\rat$-span of a Hodge class can be identified with an element $\alpha \in X^{\ast}(T^{E})$ such that   \[
-\alpha\in \Xi_{\wedge^{r} V_{1}^{\vee, \oplus m}}\ \ \text{and}\ \ \alpha\in \Xi_{(\wedge^{s}V_{2}^{\oplus n})(w-r)}.
\]
Using the polarization on $V_1$, we rewrite these conditions as
\[
\alpha\in \Xi_{\wedge^{r} V_{1}^{\oplus m}}\ \ \text{and}\ \ \alpha\in \Xi_{(\wedge^{s}V_{2}^{\oplus n})(w-r)}.
\]
 Moreover, the existence of a Hodge class in degree $2w$ on some product $A_1^m \times A_2^n$ is equivalent to the existence of  $\alpha\in X^{\ast}(T^{E})$ and $r$ with $0 \le r \le 2w$ such that $\alpha$ is a $\integ_{\ge 0}$-linear combination of the weights in $\Xi_{V_{1}}$ and 
$\alpha+(w-r)\chi$ is a $\integ_{\ge 0}$-linear combination of the weights in $\Xi_{V_{2}}$. Let $s = 2w-r$.  The choices $(r,s) = (0,2w)$ and $(r,s) = (2w,0)$ correspond to Hodge classes which come from $A_1^m$ and $A_2^n$ by pullback, while classes with $r$ and $s$ positive are conjecturally the classes of nontrivial correspondences between $A_1^m$ and $A_2^n$.
 
For such $\alpha\in \Xi_{\wedge^{r} V_{1}^{ \oplus m}}$, if moreover $\alpha=r\alpha_{0}$ for some $\alpha_{0}\in \Xi_{V_{1}}$ and some positive integer $r$, i.e., 
\[ 
\alpha\in r \Xi_{V_1} \subset r \Xi_{V_1^{\oplus m}} \subset \Xi_{\wedge^{r} V_{1}^{\oplus m}}\ \ \text{and}\ \ \alpha\in \Xi_{(\wedge^{s}V_{2}^{\oplus n})(w-r)}.
\]
we call the related Hodge classes \emph{torsion-infinite Hodge classes from $A_{1}$ to $A_{2}$}, regardless of the choice of $m$ and $n$ (and of $r$ and $s$).
 As a consequence of our definition, these classes are in $H^{r,0}(A_{1}^{m})^{\vee}\otimes H^{s}(A_{2}^{n}, \C)(w-r)$ (or $H^{0,r}(A_{1}^{m})^{\vee}\otimes H^{s}(A_{2}^{n}, \C)(w-r)$).  In particular, these classes are \emph{not} in the $\rat$-span of those classes which are pulled back from $A_1^m$ or $A_2^n$, and thus the torsion-infinite Hodge classes are extra Hodge classes. 

 \begin{prop}\label{Prop: torsion_infinite_hodge-cycles}
Let $A_1$ and $A_2$ be two isotypic CM abelian varieties over a number field $K$.  Suppose that $\operatorname{Hom}_{\bar K}(A_1, A_2) =(0)$, i.e., that  $A_{1,\bar K}$ and $A_{2,\bar K}$ have no common nontrivial isogeny factor.
 
Then $A_1$ is potentially torsion infinite for $A_2$ if and only if there is a torsion-infinite Hodge class from $A_{1}$ to $A_2$.
\end{prop}
\begin{proof}
Assume that $A_1$ is potentially torsion infinite for $A_2$ over $K$. By Theorem~\ref{Thm: main_thm_CM_pair}, 
\[
X^{\ast}(T_{\Phi_{1}})\otimes \Q\subset X^{\ast}(T_{\Phi_{2}})\otimes \Q.
\]
Suppose that $\alpha\in \Xi_{V_1}$. Then $\alpha=\sum_{\beta\in \Xi_{V_{2}}}c_{\beta} \beta$ for certain rational numbers $c_\beta$.
  Note that, if $\beta\in \Xi_{V_{2}}$, then its complex conjugate $\beta\circ c = \chi-\beta$ is in $\Xi_{V_2}$, too.
 Using this fact, we can rewrite $\alpha$ as
\[
\alpha=\left(\sum_{\beta\in \Xi_{V_{2}}}c_{\beta}^{+}\beta\right)-c_{\chi}^{+} \chi,
\]
where $c_{\beta}^{+}$ and $c_{\chi}^{+}$ are non-negative rational numbers. Choose a positive integer $m$ such that $mc_{\beta}^{+}$ and $mc_{\chi}^{+}$ are integers. Then
\begin{equation}
    \label{E:represent alpha in V2}
m\alpha=(\sum_{\beta\in \Xi_{V_{2}}}mc_{\beta}^{+} \beta) -mc_{\chi}^{+}\chi.
\end{equation}
Let $n=\sum_{\beta\in \Xi_{\rho_2}}mc_{\beta}^{+}$.

Consider the embedding $\mathbb G_m \to T^E$ induced by $\rat\hookrightarrow E$, and let $\tau$ be the standard (positive) generator of $X^*(\mathbb G_m)$.  Then $\gamma|_{\mathbb G_m} = g\tau$ for any $\gamma \in \Xi_{V_1}\cup \Xi_{V_2}$,
while $\chi|_{\mathbb G_m} = 2g\tau$.  Thus, restricting \eqref{E:represent alpha in V2} to $\mathbb G_m$ and computing coefficients of $\tau$ yields
\[
mg = ng - 2 mc_\chi^+.
\]
In particular, $m+n=n-m +2m=2mc_{\chi}^{+}+2m=2(mc_{\chi}^{+}+m)$ is even. Then $H^{m}(A_{1}^{m},\Q)\otimes H^{n}(A_{2}^{n},\Q)(mc_{\chi}^{+}+m)$ contains a torsion-infinite Hodge class from $A_{1}$ to $A_{2}$.

Conversely, if there is a torsion-infinite Hodge class from $A_{1}$ to $A_{2}$, then there exists a weight $\alpha_0\in \Xi_{V_{1}}$ which is a $\rat$-linear combination of the weights in $\Xi_{V_{2}}$. Since $\Gal(E/\Q)$ acts transitively on $\Xi_{V_{1}}$, $X^{\ast}(T_{\Phi_{1}})\otimes \Q\subset X^{\ast}(T_{\Phi_{2}})\otimes \Q $. By theorem~\ref{Thm: main_thm_CM_pair}, $A_1$ is potentially torsion infinite for $A_2$ over $K$.
\end{proof}

\begin{remark}
In the proof, we choose  $r=m$ and $s=n=\sum_{\beta\in \Xi_{\rho_{2}}}mc_{\beta}^{+}$ for convenience. However, sometimes smaller $m$ and $n$ can be chosen. See Examples \ref{Exa: torsion_infinite-7} and \ref{Exa: torsion_infinite-11}.
\end{remark}

Before displaying some concrete examples, let us prove the following lemma. Recall our discussion of nondegenerate abelian varieties (\ref{S:nondegenerate}).
 
 \begin{lemma}\label{Lem: nondegenerate}
Let $A_1$ and $A_2$ be two CM abelian varieties. Suppose that $A_i$ has a CM-type $(E, \Phi_i)$ and $A_{2}$ is nondegenerate. Then $X^{\ast}(T_{\Phi_{1}})\otimes \Q\subseteq X^{\ast}(T_{\Phi_{2}})\otimes \Q$. In particular, if $A_{1}$ and $A_{2}$ are nondegenerate, then $X^{\ast}(T_{\Phi_{1}})\otimes \Q= X^{\ast}(T_{\Phi_{2}})\otimes \Q$.
\end{lemma}

\begin{proof}
First, we let $E^+$ be the totally real subfield of $E$ and let $\mathbb{U}_1^{E^+}$ be the norm one subtorus of $T^{E^+}$. For any CM type $(E,\Phi)$, 
\[
\mathbb{U}_1^{E^+}\subset (\ker (N_{E, \Phi}))^{\circ}
\]
because the restriction of $N_{E,\Phi}$ to $T^{E^+}$ is simply $N_{E^+/\rat}$; and a dimension count shows that equality holds if and only if $(E,\Phi)$ is nondegenerate.

Under the hypotheses of the lemma, we have the following commutative diagram 
\[
\begin{tikzcd}
 0\ar[r] & X^{\ast}(T_{\Phi_{1}})\otimes\Q\ar[r, hook]& X^{\ast}(T^{E})\ar[d, equal]\otimes \Q\ar[r, two heads]&  X^{\ast}(\ker (N_{E, \Phi_{1}}))\otimes \Q\ar[r]\ar[d, two heads, "f"] & 0\\
  0\ar[r] & X^{\ast}(T_{\Phi_{2}})\otimes\Q\ar[r, hook]& X^{\ast}(T^{E})\otimes \Q\ar[r, two heads]&  X^{\ast}(\ker (N_{E, \Phi_{2}}))\otimes \Q\ar[r] & 0
\end{tikzcd}
\]
where the surjection $f$ is induced by the inclusion $(\ker (N_{E, \Phi_2}))^{\circ}=\mathbb{U}_1^{E^+}\hookrightarrow (\ker (N_{E, \Phi_1}))^{\circ}$. This implies that $X^{\ast}(T_{\Phi_{1}})\otimes \Q\subseteq X^{\ast}(T_{\Phi_{2}})\otimes \Q$. 
\end{proof}

\begin{exa}\label{Exa: torsion_infinite-7}
Let $E$ be $\Q(\zeta_{7})$. Then $\Gal(E/\Q)\cong\langle \sigma| \sigma^{6}=1\rangle$.  There are two isomorphism classes of CM types for $E$:
\begin{align*}
    \Phi_1 &= \left\{ 1, \sigma^2, \sigma^4\right\} &\Phi_2 &= \left\{1, \sigma, \sigma^2\right\}
\end{align*}
Let $A_i$ be an abelian variety with CM type $(E,\Phi_i)$.  Then $A_1$ is geometrically isogenous to the third power of an elliptic curve with CM by $\Q(\sqrt{-7})$, while $A_2$ is nondegenerate.  By Lemma~\ref{Lem: nondegenerate}, $A_1$ is potentially torsion infinite for $A_2$.  In fact, we have
\begin{align*}
      1+\sigma^{2}+\sigma^{4} & =(1+\sigma+\sigma^2)+(\sigma^2+\sigma^3+\sigma^4)+(\sigma^4+\sigma^5+1)-\chi\\
    &= \sum_{\tau\in \Phi_2} \tau + \sigma^2(\sum_{\tau\in \Phi_2} \tau)+ \sigma^4(\sum_{\tau\in \Phi_2} \tau)-\chi.  
\end{align*}
This means that $H^1(A_1,\Q)^\vee \otimes H^3(A_2,\Q)(1)$ contains a torsion-infinite Hodge class from $A_1$ to $A_2$.
\end{exa}

\begin{exa}
We return to the setting of Example~\ref{Exa: torsion_infinite-11}, with $E = \Q(\zeta_{11})$. Consider $A_2$ and $A_3$ with CM types $(E, \Phi_2=\{1, \sigma^6, \sigma^2, \sigma^3, \sigma^4\})$ and $(E, \Phi_{3}=\{1, \sigma^3, \sigma^6, \sigma^{9}, \sigma^2\})$. Since $A_2$ and $A_3$ are primitive, by Lemma \ref{Lem: nondegenerate}, $X^{\ast}(T_{\Phi_2})\otimes\Q= X^{\ast}(T_{\Phi_3})\otimes \Q$. Considering the relation between $X^{\ast}(T_{\Phi_2})$ and $X^{\ast}(T_{\Phi_3})$, we have 
\[
X^{\ast}(T_{\Phi_2})\subseteq X^{\ast}(T_{\Phi_{3}}).
\]
More precisely,
\[
\begin{split}
1+\sigma^{6}+\sigma^2+\sigma^{3}+\sigma^4=&(1+\sigma^3+\sigma^6+\sigma^{9}+\sigma^2)+(\sigma^2+\sigma^5+\sigma^{8}+\sigma+\sigma^{4})\\
&+(\sigma^{4}+\sigma^7+1+\sigma^{3}+\sigma^6) -\chi,
\end{split}
\]
and 
\[
\begin{split}
    3(1+\sigma^3+\sigma^6+\sigma^9+\sigma^2)=& 2(1+\sigma^6+\sigma^2+\sigma^3+\sigma^4)+(\sigma^7+\sigma^3+\sigma^9+1+\sigma)\\
    &+(\sigma^9+\sigma^5+\sigma+\sigma^2+\sigma^3)+2(\sigma^6+\sigma^2+\sigma^8+\sigma^9+1)\\
    &+(\sigma^3+\sigma^9+\sigma^5+\sigma^6+\sigma^7)-2\chi.
\end{split}
\]
So $A_2$ and $A_3$ are potentially torsion infinite for each other.
Moreover, $H^{1}(A_3, \Q)^{\vee}\otimes H^{3}(A_2, \Q)(1)$ contains a torsion-infinite Hodge class from $A_3$ to $A_2$, and $H^{3}(A_2^3, \Q)^{\vee}\otimes H^{7}(A_3^2, \Q)(2)$ contains a torsion-infinite Hodge class from $A_2^3$ to $A_3^2$.

\end{exa}

\subsection{Low dimension abelian varieties}\label{Sect: low dim}

In \cite[\S2]{Moonen-Zarhin-Hg-cls-av-low-dim}, Moonen and Zarhin list all the possible Hodge groups for absolute simple abelian varieties with dimension $g \leq 3$. We will follow their classification and use the notation $(g, {\rm Type})$ to denote an absolutely simple abelian variety with dimension $g$ and the indicated endomorphism type in the Albert classification. For instance, $(2, {\rm IV}(2,1))$ refers to an absolutely simple CM abelian surface. 

\begin{thm}\label{Thm: torsion_finite_low_dim}
Suppose $A$ and $B$ are absolutely simple abelian varieties over a common number field and assume that they are  non-isogenous over $\C$. Suppose that $\dim A \le \dim B \le 3$.  Then $A$ and $B$ are mutually essentially torsion finite except for the following cases:
\begin{enumerate}[label=(\alph*)]
\item $A$ is a CM elliptic curve and $B$ is of type $(3, {\rm IV}(3,1))$, i.e., $B$ is a simple CM abelian threefold. Then $B$ is  essentially torsion finite for $A$; and $A$ is potentially torsion infinite for $B$ exactly when there is an embedding of $\rat$-algebras $\End^0(A)\hookrightarrow \End^0(B)$. 
    \item  $A$ is a CM elliptic curve and $B$ is of type $(3, {\rm IV}(1,1))$. Then  $B$ is  essentially torsion finite for $A$; and $A$ is potentially torsion infinite for $B$ exactly when there is an isomorphism of $\rat$-algebras $\End^0(A) \iso \End^0(B)$.
    \item $[A,B]$ is of type $[(3,{\rm IV}(3,1)), (3,{\rm IV}(3,1))]$, i.e., both of them are CM abelian threefolds. Then the essential torsion finiteness depends on the CM-types of $A$ and $B$ as in Theorem~\ref{Thm: main_thm_CM_pair}.
\end{enumerate}
\end{thm}
\begin{proof}
Our proof contains two parts.  In the first part, we will assume that the pair $(A,B)$ is \emph{not} one of the cases (a), (b), or (c). The analysis of the special situations is carried out in the second part. 

Recall that, if 
\begin{equation}\label{Eqn: smt_prod=prod_smt}
    \smt(A\times B)=\smt(A)\times \smt(B),
\end{equation}
then $A$ and $B$ are mutually torsion finite (Corollary \ref{C:hodge group is product}).
Since  $\dim A\leq \dim B$, \eqref{Eqn: smt_prod=prod_smt} holds in each of the following cases.
\begin{enumerate}
    \item Suppose both $A$ and $B$ are of odd relative dimension, and they are not both of type IV.  Then \cite[Theorem~IA]{Ichikawa-Hg-gp-AV} states that \eqref{Eqn: smt_prod=prod_smt} holds.
    \item Suppose $A$ is a CM elliptic curve and condition (a) does not hold. Then \eqref{Eqn: smt_prod=prod_smt} is a consequence of \cite[Proposition~(3.8)]{Moonen-Zarhin-Hg-cls-av-low-dim}.  In particular, \eqref{Eqn: smt_prod=prod_smt} holds if $B$ is a surface (since a geometrically simple abelian surface in characteristic zero does not admit an action by a quadratic imaginary field -- this result of Shimura informs \cite[(2.2)]{Moonen-Zarhin-Hg-cls-av-low-dim}) or a non-type IV threefold.
    \item If $\dim A=\dim B=2$, then \eqref{Eqn: smt_prod=prod_smt} follows from \cite[(5.4),(5.5)]{Moonen-Zarhin-Hg-cls-av-low-dim}.
    \item If $\dim A=2$ and $\dim B=3$, then \eqref{Eqn: smt_prod=prod_smt} follows from \cite[Theorem(0.2)(iv)]{Moonen-Zarhin-Hg-cls-av-low-dim}.
\end{enumerate}

Hence we are left with two situations to discuss. For expository ease, we will let $A_1=A$ and $A_2=B$ in the following discussion. 

\textbf{Case 1.} 
Suppose the pair is of type [(3,IV(1,1)), (3,IV(1,1))], and that $A_1$ is potentially torsion infinite for $A_2$; we will
show that $A_1$ and $A_2$ are geometrically isogenous.  We start by
describing the Mumford--Tate groups of each $A_i$ although ultimately
we will analyze their \emph{special} Mumford--Tate groups, in order to
exploit the fact that isogenous one-dimensional algebraic tori are
actually isomorphic.  Recall that if $G$ is a reductive group with
derived group $G'$ and connected center $Z $, then $Z$ is a torus and $G$ is canonically isomorphic to $G'\times Z/(G'\cap Z)$.

For $i = 1,2$, the endomorphism algebra $F_i := \End^0(A_i)$ is an imaginary
quadratic field.  The Mumford--Tate group $G_i$ of $A_i$ is a unitary
similitude group in three variables attached to the quadratic
extension  $F_i/\rat$, which we denote
$\gu_{F_i}(3)$, and the Hodge group $sG_i$ is the unitary group
$\U_{F_i}(3)$.  The center of $G_i$ is $T^{F_i} = \res_{F_i/\rat}\gp_m$;
the connected center $Z_i$ of $sG_i$ is the norm one torus $T^{F_i,1} =\res^{(1)}_{F_i/\rat}\gp_m \iso  \U_{F_i}(1)$; and we have exact sequences
\[
\xymatrix{
1 \ar[r] & \su_{F_i}(3)\ar[d]^= \ar[r] & sG_i \iso \U_{F_i}(3) \ar@{^(->}[d] \ar[r]^{\ \ \ \ \ \ \det}
&T^{F_i,1}\ar[r] \ar@{^(->}[d]& 1\\
1 \ar[r] & \su_{F_i}(3)  \ar[r] & G_i \iso \gu_{F_i}(3) \ar[r] & T^{F_i} \ar[r] & 1}
\]
The restriction $\delta_i := \det|_{Z_i}$ is $[3]_{Z_i}$, the cubing map.
Moreover, $H_1(A_i,\rat)$ is the standard
representation of $G_i$ (see, e.g.,
\cite[(2.3)]{Moonen-Zarhin-Hg-cls-av-low-dim}).

Note that $\dim G_1 = \dim G_2$.  Under the assumption that $A_1$ is
potentially torsion infinite for $A_2$, we have $\dim G_{12} = \dim
G_2$ (Theorem \ref{T:A abs simp}).  Therefore $\dim G_{12} = \dim G_1$ as well, and thus
$A_1$ and $A_2$ are mutually potentially torsion infinite.

The isogenies $\pi_i: G_{12} \to G_i$ induce isomorphisms of Lie
algebras $\mathfrak g_{12} \to \mathfrak g_{i}$.  We thus have an
isomorphism of $\rat$-Lie algebras
$\mathfrak{gu}_{F_1}(3) \iso \mathfrak{gu}_{F_2}(3)$, and so
$F_1 \iso  F_2$.  We relabel this common quadratic field $F$ and proceed.

For each $i$, the inclusion $H_1(A_i,\rat)\hookrightarrow H_1(A_1\times A_2,\rat)$ is $F$-linear.  Therefore, we have commutative diagrams
\[
\vcenter{\xymatrix{
& G_{12} \ar@{->>}[dd]^{\pi_i} \\
T^F \ar@{^(->}[ur] \ar@{^(->}[dr] \\
& G_i
}}
\text{\quad and \quad }
\vcenter{\xymatrix{
& sG_{12} \ar@{->>}[dd]^{\pi_i} \\
T^{F,1} \ar@{^(->}[ur] \ar@{^(->}[dr] \\
& sG_i
}}
\]
where the
right-hand diagram is the restriction of the left-hand diagram to
Hodge groups.

Fix some $i$, and consider the isogeny of Hodge groups $\pi_i: sG_{12}
\to sG_i$.  Let $M_i = \pi_i^{-1}(\su_F(3))^\circ$.  Since $\su_F(3)$
is simply connected, $M_i \iso \su_F(3)$ maps isomorphically onto its
image.  Let $d_{12}\colon sG_{12} \to sG_{12}/M_i$ be the projection.  The quotient $sG_{12}/M_i$ is a smooth geometrically connected
group which is isogenous to the one-dimensional torus $T^{F,1}$, and
thus is isomorphic to $T^{F,1}$.  Similarly, the connected center $Z_{12}$ of $sG_{12}$ is isomorphic to
$T^{F,1}$, and we have a commutative diagram:
\[
\xymatrix{
 & sG_{12} \ar@{->>}[r]^{d_{12}} \ar@{->>}[dd]^{\pi_i} & T^{F,1}
 \ar@{->>}[dd]^{\bar\pi_i} \\
 T^{F,1}\ar@{^(->}[ur] \ar@{^(->}[dr] &\\
& sG_i  \ar@{->>}[r]^\det & T^{F,1} 
}
\]
Let $\delta_{12}= d_{12}|_{Z_{12}}$, and note that $\ker \delta_{12} =
Z_{12} \cap \su_F(3)$. Since $\bar\pi_i \circ \delta_{12} =
\delta_1=[3]$ and $\ker[3]$ is simple,  exactly one of $\bar\pi_i$ and $\delta_{12}$ is an isomorphism.  So either:
\begin{itemize}
    \item $\delta_{12}$ is an isomorphism.  Then $Z_{12}\cap \su_F(3)
      = \set{1}$, and so $sG_{12} \iso \su_F(3)\times \U_F(1)$, and
      $\pi_i$ is the canonical projection; or

    \item $\bar\pi_i$ is an isomorphism. Then $Z_{12}\cap \su_F(3) = \ker[3]$, $sG_{12} \iso \U_F(3)$, and $\pi_i$ is an isomorphism.
\end{itemize}
Of course, the isomorphism class of $sG_{12}$ is independent of the
choice of $i$; and we have seen that each $\pi_i$ is determined, up to
isomorphism, by the isomorphism class of $sG_{12}$.  Therefore,
$sG_{12} \to sG_i \to \GL_{V_i}$ is independent of $i$, and so $A_1$
and $A_2$ are isogenous (Lemma \ref{L:MT isomorphic}).  (After the fact, using Lemma \ref{L:MT isomorphic}(c), we
recognize that the second case happens, i.e., that $sG_{12} \iso U_F(3)$.)

\textbf{Case 2.} If the pair is of type $[(3, {\rm IV}(1,1)), (3, {\rm IV}(3,1))]$, then the Hodge group of $A_1$ has been explained in Case 1.  Note in particular that the center of $\smt(A_1)$ is $U_{F_1}(1)$, a one-dimensional torus.

Now consider $A_2$.  It has complex multiplication by a CM field $E_2$.  Since $\dim A_2 = 3$ is prime, the CM type is nondegenerate, i.e., $\dim \smt(A_2) = 3$.

In particular, there is no isogeny from the center of $\smt(A_1)$ to $\smt(A_2)$.  By \cite[Lemma 3.6]{Moonen-Zarhin-Hg-cls-av-low-dim}, $A_1$ and $A_2$ satisfy \eqref{Eqn: smt_prod=prod_smt}, and thus are mutually essentially torsion finite. This finishes the first part of the proof.

It remains to discuss cases (a), (b), and (c).  Of course, there is nothing to prove for case (c). As for (a) and (b), since 
\[\dim \mt(A_1\times A_2)\geq \dim \mt(A_2)>2=\dim \mt(A_1),\]
we immediately deduce that $A_2$ is essentially torsion finite for $A_1$. 

Moreover, by \cite[Proposition~(3.8)]{Moonen-Zarhin-Hg-cls-av-low-dim}, $\smt(A_1\times A_2) = \smt(A_1)\times \smt(A_2)$ if and only if there is no embedding $\End^0(A_1)\hookrightarrow \End^0(A_2)$; and this is equivalent to the essential torsion finiteness of $A_1$ for $A_2$ (Corollary~\ref{C: TF cm elliptic curve}).
 \end{proof}

\bibliographystyle{amsalpha}
\bibliography{draft}

\end{document}